\documentclass[a4paper, 12pt]{amsart}
\numberwithin{equation}{section}
\usepackage{amssymb}
\usepackage{amsmath}
\usepackage{amsfonts}
\usepackage{marvosym}
\usepackage{tikz}
\newcommand*\circled[1]{\tikz[baseline=(char.base)]
{\node[shape=circle,draw,inner sep=1] (char) {#1};}}

\newtheorem{theorem}{Theorem}[section]
\newtheorem{thm}{Theorem}[section]
\newtheorem{lem}[theorem]{Lemma}

\newtheorem{prop}[theorem]{Proposition}
\newtheorem{proposition}[theorem]{Proposition}

\newtheorem{defin}[theorem]{Definition}
\newtheorem{corollary}[theorem]{Corollary}
\theoremstyle{definition}
\newtheorem{definition}[theorem]{Definition}
\newtheorem{rem}[theorem]{Remark}
\newtheorem{remark}[theorem]{Remark}

\newcommand{\spa}{\operatorname{span}}
\newcommand{\ftp}{\, \circled{\rm{{\tiny F}}}\, }

\newtheorem{example}[theorem]{Example}

\def\ca{{\mathcal A}}
\def\cb{{\mathcal B}}

\def\cd{{\mathcal D}}

\def\ch{{\mathcal H}}

\def\cam{{\mathcal M}}

\def\cs{{\mathcal S}}

\def\ga{{\mathfrak A}}
\def\gb{{\mathfrak B}}

\def\gam{{\mathfrak M}}
\def\gn{{\mathfrak N}}\def\gpn{{\mathfrak n}}

\def\bc{{\mathbb C}}

\def\bm{{\mathbb M}}
\def\bn{{\mathbb N}}

\def\bz{{\mathbb Z}}

\def\a{\alpha}
\def\b{\beta}
\def\g{\gamma}  \def\G{\Gamma}
\def\d{\delta}  
\def\eeps{\epsilon}
\def\eps{\varepsilon}

\def\l{\lambda} 
\def\k{\kappa}
\def\m{\mu}

\def\n{\nu}
\def\r{\rho}

\def\f{\varphi}  \def\F{\Phi}
\def\th{\theta} 
\def\om{\omega}

\def\id{\hbox{id}}

\def\aut{\mathop{\rm aut}}
\def\carf{\mathop{\rm CAR}}

\def\tr{\mathop{\rm Tr}}

\newcommand{\ty}[1]{\mathop{\rm {#1}}}

\def\id{{\rm id}}

\def\ad{\mathop{\rm ad}}

\def\idd{{1}\!\!{\rm I}}

\newcommand{\nn}{\nonumber}

\DeclareMathAlphabet{\mathpzc}{OT1}{pzc}{m}{it}

\begin{document}
\date{\today}
\title[fermi systems, detailed balance]{$C^*$-fermi systems and detailed balance}
\author[Crismale, Duvenhage, Fidaleo]
       {Vitonofrio Crismale, Rocco Duvenhage, and Francesco Fidaleo}
\address{Dipartimento di Matematica, Universit\`{a} degli studi di Bari, Bari, Italy}
\email{{\tt vitonofrio.crismale@uniba.it}}
\address{Department of Physics, University of Pretoria, Pretoria, South Africa}
\email{{\tt rocco.duvenhage@up.ac.za}}
\address{Dipartimento di Matematica, Universit\`{a} degli studi di Roma Tor Vergata,
Rome, Italy}
\email{{\tt fidaleo@mat.uniroma2.it}}

\begin{abstract}
A systematic theory of product and diagonal states is developed for tensor products of $\bz_2$-graded $*$-algebras, as well as $\bz_2$-graded $C^*$-algebras. As a preliminary step to achieve this goal, we provide the construction of a {\it fermionic $C^*$-tensor product }of $\bz_2$-graded $C^*$-algebras.
Twisted duals of positive linear maps between von Neumann algebras are then
studied, and applied to solve a positivity problem on the infinite Fermi lattice. Lastly, these results are used to define fermionic detailed balance (which includes the definition for the usual tensor product as a particular case) in
general $C^*$-systems with gradation of type $\bz_2$, by viewing such a system
as part of a compound system and making use of a diagonal state.

\vskip0.1cm\noindent \\
{\bf Mathematics Subject Classification}: 46L06, 46L55, 81R15, 82B10.\\
{\bf Key words}: $\bz_2$-graded $C^*$-algebras, $\bz_2$-graded tensor products, product states, diagonal states, twisted duals, detailed balance.
\end{abstract}

\maketitle

\tableofcontents

\section{Introduction}

Detailed balance is a central topic in Statistical Mechanics. In the quantum setting, it has been extensively studied and developed, see \emph{e.g.} \cite{Ag, Al, CW, KFGV, M}.
However, to our best knowledge, no formulations specific to graded algebras, which include fermionic systems, are present in literature.

The goal of this paper is to develop a general theory of product systems (or compound systems)
endowed with a gradation of type $\bz_2$, and to use this theory as an abstract framework
for formulating detailed balance aimed specifically at fermionic systems.

A central notion in this respect is a tensor product for two $\bz_2$-graded $*$-algebras,
which we refer to as the Fermi tensor product. We consider this product in a very general algebraic point of view,
and successively study its completion in a natural norm, which is introduced here, corresponding to the maximal $C^*$-cross norm for the usual tensor product.

A crucial issue in product systems is the product of states. In the special case of Canonical
Anticommutation Relations (CAR for short) algebras, which reduce to the $C^*$-completion of the infinite
tensor product of $2\times 2$ matrices (see Section \ref{AfdZ2}), the product state was introduced
and studied in \cite{AM1,AM2}, where the reader is referred also for some relevant applications.
These states have a central role in the study of distributional symmetries. Indeed, in \cite{CF0,CF1}
it has been stated that they are exactly the ergodic normalised positive functionals invariant under
the action of the infinite symmetric group on the Fermi $C^*$-algebra. The theory of product states for
the Fermi tensor product, which includes the CAR algebra as a special case, is therefore developed here along with their GNS representation.

Equally important is the diagonal state, which is a key construction in the theory
of joinings and related ideas in ergodic theory. Diagonal states for the usual algebraic
tensor product of von Neumann algebras have been studied and used with much success in
noncommutative ergodic  theory, see in particular \cite{BCM, BCM2, D, D2, D3, Fid,F27a}.
They are also important in quantum detailed balance, where they provide an abstract
version of maximally entangled states, see in particular \cite{DS1, DS2}.
It is therefore to be expected that a fermionic version of diagonal states,
formulated in terms of the Fermi product, is crucial in the context of detailed balance for fermionic systems.

In addition, a version of quantum detailed balance, tailored to systems consisting
of indistinguishable fermions, was recently introduced and studied in \cite{Dfer} for
the special case of a finite lattice. On the one hand, this confirmed the
relevance of fermionic entangled states in the present context. On the other
hand, it equally serves as a further and direct motivation for an abstract version of diagonal states for the Fermi tensor product, which is developed in this paper.

Moreover, in \cite{Dfer}, in the case
of the finite dimensional observable algebra associated to a finite fermion
lattice, initial steps were taken
to develop a duality theory for dynamical maps. This led to a characterization of fermionic standard
quantum detailed balance in terms of the fermionic dual of the dynamics.

One remaining problem there was to determine if the fermionic dual of a dynamical map has
positivity properties corresponding to that of the original map. It is one of the
objectives of this paper to solve this problem, even in a more general form,
by developing a duality theory in the von Neumann algebraic framework. Indeed, it will be seen that
the dual map has the same positivity properties (positivity,
$n$-positivity, or complete positivity) as the dynamical map itself, if the
latter is even.

The final part of this work is devoted to
formulate and discuss fermionic standard quantum detailed balance in the
general $C^*$-algebra setup, focussing on conceptual aspects of the
mathematical formulation. In particular, the Fermi tensor product clarifies
the analogy with standard quantum detailed balance expressed in terms of the
usual tensor product.

As quantum detailed balance is the origin of this paper, it is worth
making some brief general remarks about this topic.
The fermionic quantum detailed
balance condition we are interested in, that is fermio\-nic standard quantum
detailed balance, is most directly motivated by standard quantum detailed
balance with respect to a reversing operation. The latter condition has been studied
in \cite{FU, FR, BQ2, DS1}. Closely related work has appeared in
\cite{BQ, DF, MS, R}. All of these in turn are built on the early works on quantum
detailed balance. We mention \cite{Ag}, where to our knowledge this notion was first introduced for the case of open quantum Markov systems. Agarwal's approach was later developed and extended in \cite{CW,M}, whereas in \cite{Al,KFGV} the authors introduced perhaps the best known definition of quantum detailed balance for quantum dynamical semigroups. We mention also the case in which the notion is strictly connected to symmetry properties related to the KMS condition \cite{GL}.

As previously stressed, none of these
references attempted to set up detailed balance specifically for systems of
indistinguishable fermions. On the other hand, Markov chains and related
matters ({\it e.g.} \cite{AFM, Fid2}) and also  the extension of the abstract theory of disordered systems ({\it cf.} \cite{BF})
have been treated in the context of fermions,
but without reference to detailed balance designed
specifically for indistinguishable particles.

To set up the abstract algebraic framework, we work in terms of $\mathbb{Z}_{2}$-gradings of $*$-algebras, and exploit the resulting
Klein transformations. This includes the notion of twisted commutant in
the case of von Neumann algebras, a concept developed in \cite{DHR1}, and subsequently used in \cite{DHR2}. Our approach allows us
to move from the concrete realm of the above cited paper \cite{Dfer} to an abstract framework for fermionic systems, where one can
define a general version of the fermionic dual, that is the twisted
dual, of a positive linear map.

In \cite{DHR1}, twisted duality for an algebra of fields was studied in terms
of the twisted commutant. This is related to locality, or more precisely Haag
duality (see \cite{HS} and \cite {H}, Section III.4.2) in the so-called {\it algebraic quantum
field theory}, when Fermi fields are involved. In this paper, we define and study the twisted dual of a dynamical map, and more generally of
a positive linear map from one von Neumann algebra to another.

We mention that twisted duality in the sense of \cite{DHR1} was shown for the CAR algebra case in
\cite{F}. The reader is referred to \cite{W}, Section 13 for a result of this type, and
\cite{S} for further related work.

However, a particularly useful approach to
this problem was followed in \cite{BJL}, the main result of which will be
applied in Section \ref{AfdTralie} in order to connect our general
results to the concrete case of the
CAR algebra, and contribute to solve the aforementioned problem.

The plan of the paper is as follows. Apart from the preliminary Section  \ref{AfdPre}, in Section \ref{gienes} we collect some general results on the (analogue of the) Gelfand-Naimark-Segal (GNS for short) representation arising from a positive functional on a merely ({\it i.e.} without any topology)
involutive algebra, which is useful in the sequel.

In Section \ref{AfdZ2}, we recall the main facts concerning $G$-graded involutive algebras, focussing ourselves on the case of our interest $G=\bz_2$.

One of the main ingredients to treat Fermi systems, is the Jordan-Klein-Wigner transformation (simply mentioned as the Klein transformation). After recalling its original definition, in Section \ref{kjw} we investigate it in some detail from an abstract and concrete point of view. Related concepts, in particular the twisted
commutant of a von Neumann algebra, are also discussed.

In Section \ref{tensor}, we recall the construction of the $\bz_2$-graded (Fermi)
tensor product of $\bz_2$-graded $*$-algebras, whereas in Section \ref{AfdProdToest} we study the product functional of two states on the algebraic Fermi tensor product and show whether it is positive. Indeed, given two states,
each of them on the $\bz_2$-graded $*$-algebra, we  define their product functional
on the above tensor product, and prove that it is positive if at least one of the given states
is even. This is a generalization of Theorem 1 in \cite{AM1}.

Furthermore, in Section \ref{fermicstar} the maximal, possibly extended-valued,
seminorm on the algebraic Fermi tensor product is introduced, and it is proved that it is indeed a $C^*$-norm.
This gives rise to the Fermi $C^*$-product, analogous to the completion of the usual tensor product w.r.t. the maximal $C^*$-cross norm.

Even though it is not used in the rest of the paper, in Section \ref{AfdProdGNS} we also provide the description of the GNS representation for any product state, which is indeed more involved than the usual one, and is achieved via the Stinespring dilation.

Section \ref{AfdDiag} is mainly devoted to introduce the diagonal state for the Fermi tensor product,
and its GNS representation is investigated as well.

Duality is then studied
in Section \ref{AfdDualiteit}, where we obtain an abstract result
required to solve our positivity problem from
\cite{Dfer} mentioned earlier.

As already pointed out, a duality theory is developed for positive
linear maps between von Neumann algebras, by defining and investigating the
twisted duals of such maps in terms of the twisted commutants of the von
Neumann algebras, with particular emphasis on the positivity properties of the
twisted dual maps. The positivity problem for an infinite Fermi lattice is then solved in Section \ref{AfdTralie}.

The paper concludes with a proposal for fermionic detailed
balance in the abstract framework in Sections \ref{AfdFfb} and \ref{AfdCffb}.
Much of the foregoing theory is applied there in order to motivate the proposal both conceptually and technically, along with illustration by some examples.

\section{Preliminaries}
\label{AfdPre}

Let $\ga$ be an involutive, or equivalently, a $*$-algebra. By $\aut(\ga)$ we denote the group of its $*$-automorphisms.

For two linear spaces $X$ and $Y$, we denote by $X\dot{+}Y$ and $X\odot Y$ their algebraic direct sum and tensor product, respectively.
If in addition, $\ga$ and $\gb$ are involutive algebras with $*$ denoting their involution, then $\ga\otimes\gb$ will denote the algebraic tensor product $\ga\odot\gb$ equipped with the usual product and involution given on the generators by
$$
(a_1\otimes b_1){\bf\cdot}(a_2\otimes b_2)=a_1a_2\otimes b_1b_2\,,\quad (a_1\otimes b_1)^\dagger=a_1^*\otimes b_1^*\,,
$$
for all $a_1\,a_2\in\ga, b_1,b_2\in\gb\,.$\footnote{We introduce the symbols ${\bf\cdot}$ and ${}^\dagger$ to denote the product and the involution in
$\ga\otimes\gb$ in order to distinguish them from the analogous operations (denoted by the standard symbology)
in the $\bz_2$-graded tensor product $\ga\, \circled{\text{{\tiny F}}}\, \gb$ whenever $\ga$ and $\gb$ are equipped with a $\bz_2$-grading.}

If $\{\ga_\iota\}_{\iota\in I}$ is a collection of $C^*$-algebras indexed by a set $I$, $\oplus_{\iota\in I}\ga_\iota$
denotes their $C^*$-direct sum as defined in Section L2 of \cite{WO}. It is nothing but the
$C^*$-completion of $\dot{+}_{\iota\in I}\ga_\iota$. If $|I|<+\infty$,
then $\dot{+}_{\iota\in I}\ga_\iota=\oplus_{\iota\in I}\ga_\iota$ at the level of involutive algebras, where $|\cdot|$ denotes the cardinality.

For $C^*$-algebras $\ga$ and $\gb$, we denote by $\ga\otimes_{\rm max}\gb$
and $\ga\otimes_{\rm min}\gb$ the completion of $\ga\otimes\gb$ w.r.t. the maximal and minimal $C^*$-cross norm, respectively, see {\it e.g.} \cite{T}.

For states $\om\in\cs(\ga)$, $\f\in\cs(\gb)$, we denote by
\[
\psi_{\om,\f}\in\cs(\ga\otimes_{\rm min}\gb)
\]
the product state on the $C^*$-algebra $\ga\otimes_{\rm min}\gb$.
A fortiori, $\psi_{\om,\f}$ is also well defined as a state on $\ga\otimes_{\rm max}\gb$. Therefore, with an abuse of notation we write $\psi_{\om,\f}\in\cs(\ga\otimes_{\rm max}\gb)$, or merely
$\psi_{\om,\f}\in\cs(\ga\otimes\gb)$.

Let $\ga$ be a $C^*$-algebra, and $\f\in\cs(\ga)$ a state.
By $\big(\ch_\f,\pi_\f,\xi_\f\big)$, we denote the Gelfand-Naimark-Segal (GNS for short) representation associated to the state $\f$, see {\it e.g.} \cite{T}. If in addition $\th\in\aut(\ga)$ is a $*$-automorphism leaving invariant the state $\f$, then there exists a unitary $V_{\f,\th}$ acting on $\ch_\f$ which implements $\th$, that is
$$
V_{\f,\th}\pi_\f(a)V_{\f,\th}^*=\pi_\f(\th(a))\,,\quad a\in\ga\,.
$$
The quadruple $\big(\ch_\f,\pi_\f, V_{\f,\th},\xi_\f\big)$ is called {\it the covariant GNS representation} associated to the triple $(\ga,\th,\f)$.

If $\ga$ is a $C^*$-algebra with a $C^*$-subalgebra $\gb$,
the linear mapping $E:\ga\to \gb$ is called a projection if $E(b)=b$ for all $b\in\gb$,
and is said to be a $\gb$-bimodule map if $E(ab)=E(a)b$, and $E(ba)=bE(a)$ for all $a\in\ga, b\in\gb$.
A positive $\gb$-bimodule projection $E$ is called a \emph{conditional expectation}.

If $\Phi: \ga\mapsto \gb$ is a linear map between the $C^*$-algebras $\ga$ and $\gb$, it is said to be {\it completely positive} if all maps
\begin{equation*}
\Phi\otimes\id_{\bm_n(\bc)}:\bm_n(\ga)\to\bm_n(\gb)\,,\quad n=1,2,\dots
\end{equation*}
are positive.

\section{The GNS representation for algebraic probability spaces}
\label{gienes}

We recall and extend some properties concerning the Gelfand-Naimark-Segal representation associated to a so-called algebraic probability space.

An {\it algebraic probability space} is a pair $(\ga,\f)$ where $\ga$ is an involutive algebra,
and $\f$ is a positive linear functional.\footnote{A positive linear functional $\f$ on the involutive algebra $\ga$
is an element of the algebraic dual of $\ga$ assuming positive values on positive elements:
$$
\f(a^*a)\geq0\,,\quad a\in\ga\,.
$$}

We report without proof (see {\it e.g.} \cite{T}) the following well-known result.
\begin{prop}
\label{cbs}
Let $\f$ be a positive functional on the involutive algebra $\ga$. Then the sesquilinear form $(x,y)\in\ga\times \ga\mapsto\f(x^*y)\in\bc$
\begin{itemize}
\item[(i)] is hermitian: $\f(y^*x)=\overline{\f(x^*y)}$,
\item[(ii)] and satisfies the Cauchy-Bunyakovsky-Schwarz inequality:
$$
|\f(y^*x)|^2\leq\f(x^*x)\f(y^*y)\,.
$$
\end{itemize}
\end{prop}
By Proposition \ref{cbs}, we first see that $\ga$ is equipped with the semi-inner product $(x,y)\mapsto\f(y^*x)$, with seminorm $\|x\|:=\f(x^*x)^{1/2}$.
Let
$$
\gpn_\f:=\{x\in\ga\mid\f(x^*x)=0\}
$$
be the left ideal associated to $\f$, and denote by $\ch_\f$ the completion of the quotient space $\ga/\gpn_\f$ w.r.t. the seminorm $\f(x^*x)^{1/2}$. As usual,
$a\in\ga\mapsto a_\f\in\ch_\f$ denotes the canonical quotient map.

In addition, it is matter of routine to check that, for each $a,x\in\ga$,
$$
\pi^o_\f(a)x_\f:=(ax)_\f
$$
uniquely defines linear operators on the common dense domain
$$
\cd_{\pi^o_\f(a)}=\{x_\f\mid x\in\ga\}=:\cd_\f\,.
$$
The main properties of $\pi^o_\f$ are summarised in the following:
\begin{thm}
	\label{gnsrcfp}
	Let $\f$ be a positive linear functional on the involutive algebra
	$\ga$. With the above notations, the following hold true.
	\begin{itemize}
		\item[(i)] On $\cd_\f$, we have for $a,b\in\ga$ and $\a,\b\in\bc$,
		$\pi^o_\f(\a a+\b b)=\a\pi^o_\f(a)+\b\pi^o_\f(b)$,
		$\pi^o_\f(ab)=\pi^o_\f(a)\pi^o_\f(b)$.
		\item[(ii)] For $x\in\ga$, $\pi^o_\f(x)^*\supset\pi^o_\f(x^*)$ and
		therefore the linear operators $\big\{\pi^o_\f(x)\mid x\in\ga\big\}$ are
		closable on the common core $\cd_\f$.
		\item[(iii)]
		$\overline{\pi^o_\f(x^*)}\subset\left(\overline{\pi^o_\f(x)}\right)^*$
		and therefore $\overline{\pi^o_\f(x)}$ is hermitian, provided $x=x^*$.
		\item[(iv)] For $\pi_\f(x):=\pi^o_\f(x^*)^*$ we have
		$\pi_\f(x)^*\subset\pi_\f(x^*)$, and therefore $\pi_\f(x)^*$ is
		Hermitian, provided $x=x^*$.
		\item[(v)] For $x\in\ga$, if for some constant $A_x$ we have
		$\f(a^*x^*xa)\leq A_x\f(a^*a)$ for all $a\in\ga$, then
		$\cd_{\overline{\pi^o_\f(x)}}=\ch_\f$ and therefore
		$\overline{\pi^o_\f(x)}$ is bounded.
		\item[(vi)] If for the uniform constant $C$ we have $|\f(x)|\leq
		C\f(x^*x)^{1/2}$ for all $x\in\ga$, then there exists
		$\xi_\f\in\bigcap_{x\in\ga}\cd_{\pi^o_\f(x^*)^*}$ which is cyclic, i.e.
		$\overline{\pi^o_\f(\ga)\xi_\f}=\ch_\f$, and such that
		$$
		\f(x)=\big\langle\pi^o_\f(x^*)^*\xi_\f,\xi_\f\big\rangle\,,\quad
		x\in\ga\,.
		$$
	\end{itemize}
\end{thm}
\begin{proof}
	(i) is trivial. For (ii), we have
	\begin{align*}
	\langle\pi^o_\f(x)a_\f,b_\f\rangle=&\langle(xa)_\f,b_\f\rangle=\f(b^*xa)\\
	=&\f((x^*b)^*a)=\langle a_\f,(x^*b)_\f\rangle\\
	=&\langle a_\f,\pi^o_\f(x^*)b_\f\rangle\,,
	\end{align*}
	and therefore $\pi^o_\f(x)^*$ extends $\pi^o_\f(x^*)$. Since
	$\pi^o_\f(x)$ admits the closed extension $\pi^o_\f(x^*)^*$, it is
	closable.
	
	Concerning (iii),
	$$
	\left(\overline{\pi^o_\f(x^*)}\right)^*=\pi^o_\f(x^*)^*\supset\overline{\pi^o_\f(x)}\,.
	$$
	
	Concerning (iv), by (ii), we get
	\begin{align*}
	\pi_\f(x^*)=\pi^o_\f(x)^*\supset\overline{\pi^o_\f(x^*)}=\pi^o_\f(x^*)^{**}=\big(\pi^o_\f(x^*)^*\big)^*=\pi_\f(x)^*\,.
	\end{align*}

	If (v) is satisfied for some $x\in\ga$, then $\pi^o_\f(x)$ is bounded on
	the dense domain $\cd_\f$. Therefore, it uniquely extends to a bounded
	operator on the whole $\ch_\f$, which coincides with its closure by the
	closed graph theorem.
	
	Suppose that (vi) holds true. Then $x_\f\in\ch_\f\mapsto\f(x)$ uniquely
	extends to a bounded linear functional on $\ch_\f$ and therefore, by the
	Riesz theorem, $\f(x)=\langle x_\f,\xi_\f\rangle$ for a uniquely
	determined vector $\xi_\f\in\ch_\f$. In addition,
	$$
	\langle
	\pi^o_\f(x^*)a_\f,\xi_\f\rangle=\langle(x^*a)_\f,\xi_\f\rangle=\f(x^*a)=\langle
	a_\f,x_\f\rangle\,,
	$$
	and so $\xi_\f\in\cd_{\pi^o_\f(x^*)^*}$ with
	$\pi^o_\f(x^*)^*\xi_\f=x_\f$. By (ii) this gives that $\xi_\f$ is
	cyclic, and we get
	$$
	\f(x)=\langle
	x_\f,\xi_\f\rangle=\langle\pi^o_\f(x^*)^*\xi_\f,\xi_\f\rangle\,.
	$$
\end{proof}

In the general setting previously described, by (vi) of Theorem \eqref{gnsrcfp},
the candidate for the GNS representation associated
to the algebraic probability space $(\ga,\f)$ can be viewed as the collection of
the closed operators $\{\pi^o_\f(x^*)^*\mid x\in\ga\}$, acting on the Hilbert space $\ch_\f$.

We note that, without assuming further conditions, such operators of the GNS representation
associated to an algebraic probability space do not enjoy the good well-known properties of
the usual situation arising in the $C^*$-algebraic setting. However, for the purpose of the present paper, we only deal with cases
for which (v) and (vi) in the previous theorem are satisfied
for each $a\in\ga$.\footnote{The properties (v) and (vi) of Theorem \ref{gnsrcfp} hold true
for each $a\in\ga$, provided $\ga$ is an involutive Banach algebra equipped with a bounded approximate unity,
and therefore when $\ga$ is a $C^*$-algebra, see {\it e.g.} \cite{T}}.
In this situation, the GNS representation is commonly denoted by the triple
$(\ch_\f,\pi_\f,\xi_\f)$, uniquely determined up to unitary equivalence, see {\it e.g.} \cite{T}, Section I.9.

\section{$\bz_2$-graded $*$-algebras}
\label{AfdZ2}

For $\bz_2=\{1,-1\}$ with the product as the group operation, the involutive algebra $\ga$ is an {\it involutive $\bz_2$-graded algebra} if
$$
\ga=\ga_1\dot{+}\ga_{-1}
$$
and
$$
(\ga_i)^*=(\ga^*)_i\,,\,\,
\ga_i\ga_j\subset\ga_{ij}\,,\quad i,j=1,-1\,.\footnote{Another standard terminology would be \it{involutive superalgebra}.}
$$

More generally, for a group $G$ an involutive $G$-graded algebra $\ga$ would be an algebraic direct sum
$$
\ga=\dot{+}_{g\in G}\ga_g\,,
$$
equipped with a product and an involution
satisfying
$$
\ga_g\ga_h\subset\ga_{gh}\,,\,\,\,
\ga_g^*\subset\ga_{g^{-1}}\,,\quad g,h\in G\,.
$$
The linear subspaces $\ga_g$, for $g\in G$, are called {\it  the homogeneous components} of $\ga$.
Elements of $\ga_g$, for any $g\in G$, are correspondingly called {\it homogeneous elements} of $\ga$.
For the unit $e\in G$, the homogeneous component $\ga_e$ is an involutive subalgebra of $\ga$ containing the identity $\idd_\ga$ if $\ga$ is unital.

Since we only deal with $\bz_2$-graded algebras, we do not pursue the situation associated to a general group $G$ any further.
In the case of $G=\bz_2$, for each homogeneous element $x\in\ga_{\pm 1}$ we correspondingly indicate its {\it grade} by
\[
\partial(x)=\pm1.
\]

Suppose now we have an involutive $\bz_2$-graded algebra $\ga=\ga_1\dot{+}\ga_{-1}$ as above, and define $\th:\ga\to\ga$ as
\begin{equation*}
\th\lceil_{\ga_1}=\id_{\ga_1}\,,\quad \th\lceil_{\ga_{-1}}=-\id_{\ga_{-1}}\,.
\end{equation*}
It is almost immediate to check that $\th\in\aut(\ga)$ is an involutive $*$-automorphism ({\it i.e.} $\th^2=\id_\ga$).
Conversely, let $\th\in\aut(\ga)$ such that $\th^2=\id_\ga$, and consider
\begin{align*}
\eps_1:=\frac{1}{2}(\id_{\ga}+\th)\,,&\quad \eps_{-1}:=\frac{1}{2}(\id_{\ga}-\th)\,,\\
\ga_1:=\eps_1(\ga)\,,&\quad\ga_{-1}:=\eps_{-1}(\ga)\,.
\end{align*}
\begin{lem}
With the above notations, the $*$-automorphism $\th$ induces a structure of involutive $\bz_2$-graded algebra on $\ga$.
\end{lem}
\begin{proof}
We have only to show $\ga_1\cap\ga_{-1}=\{0\}$, the others properties being trivial. Indeed, let us take $a,b,c\in\ga$ such that
$$
\frac{b+\th(b)}2=a=\frac{c-\th(c)}2\,.
$$
Then $\th(a)=a=-\th(a)$, which gives $a=0$.
\end{proof}
As a consequence, for an involutive algebra $\ga$ a $\bz_2$-grading
is always induced by an involutive $*$-automorphism, and therefore we can state the following
\begin{defin}
A $\bz_2$-graded involutive algebra is a pair $(\ga,\th)$, with $\ga$ an involutive algebra and $\th\in\aut(\ga)$ such that $\th^2=\id_\ga$.
\end{defin}
\vskip.3cm
We then also refer to $\theta$ as a $\bz_2$-grading of $\ga$, and denote
by $\ga_+:=\ga_1$ the {\it even part} (which is indeed a $*$-subalgebra
of $\ga$), and by $\ga_-:=\ga_{-1}$ the {\it odd part} of $\ga$, respectively.
As a consequence, for any $a\in\ga$, we can write $a=a_++a_-$, with $a_+\in\ga_+$, $a_-\in\ga_-$,
and this decomposition is unique. Moreover, one has $\th(a_+)=a_+$, $\th(a_-)=-a_-$.

Any involutive algebra is endowed with the trivial $\bz_2$-grading induced by $\th=\id_\ga$,
if no nontrivial $*$-automorphism is selected. In this case, $\ga_+=\ga$ and $\ga_-=\{0\}$.

In the sequel, we will not indicate the grading automorphism whenever the latter is fixed, without risk of confusion.

Let $\big(\ga^{(i)},\th^{(i)}\big)$, $i=1,2$, be a pair of $\bz_2$-graded involutive algebras,
together with a map $T:\ga^{(1)}\to\ga^{(2)}$. $T$ is said to be {\it even} if it is grading-equivariant:
$$
T\circ\th^{(1)}=\th^{(2)}\circ T\,.
$$
When $\th^{(2)}=\id_{\ga^{(2)}}$, the map $T:\ga^{(1)}\to\ga^{(2)}$ is even if and only if it is grading-invariant, that is
$T\circ\th^{(1)}=T$. As a particular case when $\big(\ga^{(1)},\th^{(1)}\big)=(\ga,\th)$
and $\big(\ga^{(2)},\id_{\ga^{(2)}}\big)=\big(\bc,\id_\bc\big)$, a functional $f:\ga\to\bc$
is even if and only if $f\circ\th=f$. If the map $T$ (or the functional $f$) is $\bz_2$-linear,
then it is even if and only if $T\lceil_{\ga^{(1)}_-}=0$ ($f\lceil_{\ga_-}=0$).

We now specialise the situation to the following example.
It comes directly from quantum field theory and statistical mechanics, and is based on the Canonical Anticommutation Relations algebra on a countable set of indices.

Recall that for an arbitrary set $L$, the {\it Canonical Anticommutation Relations} (CAR
for short) algebra over  $L$ is the $C^{*}$-algebra $\carf(L)$
with the identity $\idd$ generated by the set $\{a_j,
a^{\dagger}_j\mid j\in L\}$ ({\it i.e.} the Fermi annihilators and
creators respectively), and the relations
\begin{equation*}
(a_{j})^{*}=a^{\dagger}_{j}\,,\,\,\{a^{\dagger}_{j},a_{k}\}=\d_{jk}\idd\,,\,\,
\{a_{j},a_{k}\}=\{a^{\dagger}_{j},a^{\dagger}_{k}\}=0\,,\,\,j,k\in L\,.
\end{equation*}
A $\bz_{2}$-grading is induced on $\carf(L)$ by the $*$-automorphism $\th$ acting on the generators as
$$
\th(a_{j})=-a_{j}\,,\quad \th(a^{\dagger}_{j})=-a^{\dagger}_{j}\,,\quad
j\in L\,.
$$
Notice that, by definition,
\begin{equation}
\label{loc}
\carf(L)=\overline{\carf{}_o(L)}\,,
\end{equation}
where
\begin{equation}
\label{loc2}
\carf{}_o(L):=\bigcup\{\carf(I)\mid\, I\subset L,\text{finite}\,\}
\end{equation}
is the dense subalgebra of the {\it localised elements}.

For a countable index set $L\sim\bn$, $\carf(L)$ can be identified with the $C^{*}$-infinite tensor product of $L$-copies of $\bm_{2}(\bc)$, that is
\begin{equation}
\label{jkw}
\carf(\bn)\sim\overline{\bigotimes_{\bn}\bm_{2}(\bc)}^{C^*}\,.
\end{equation}
As shown in \cite{T}, Exercise XIV.1, the above isomorphism is achieved by a Jordan-Klein-Wigner
transformation. We briefly report such a construction for the
convenience of the reader. Indeed, consider $U_{j}:=a_{j}a_{j}^{\dagger}-a_{j}^{\dagger}a_{j}$, $V_{0}:=\idd$, and ${
V_{j}:=\prod_{n=1}^{j}U_{n}}$, for $j=1,2,\dots$\,. The Jordan--Klein--Wigner construction is defined as follows
\begin{align*}
&e_{11}(j):=a_{j}a_{j}^{\dagger}\,,\quad e_{12}(j):=V_{j-1}a_{j}\,,\nn\\
&e_{21}(j):=V_{j-1}a_{j}^{\dagger}\,,\quad e_{22}(j):=a_{j}^{\dagger}a_{j}\,,
\end{align*}
and provides a system of
commuting matrix units $\{e_{kl}(j)\,|\,k,l=1,2\}_{j\in \bn}$ in $\carf(\bn)$.
Now fix a set of points $[1,l]\subset\bn$ and consider the system of matrix units localised in $r\in\bn$, that is
$$
\{\eps_{i_rj_r}(r)\}_{i_rj_r=1,2}\subset\overline{\bigotimes_{\bn}\bm_2(\bc)}^{C^*}\,.
$$
It turns out that the map
\begin{equation*}
e_{i_1j_1}(1)\cdots e_{i_lj_l}(l)\mapsto
\eps_{i_1j_1}(1)\otimes\cdots\otimes\eps_{i_lj_l}(l)
\end{equation*}
where $i_k,j_k=1,2$, $k=1,2,\dots,l$ and $l\in\bn$, realises the isomorphism \eqref{jkw} as a consequence of \eqref{loc} and \eqref{loc2}.

We also recall that in this case the Fock representation of the CAR algebra is realised
on the anti-symmetric Fock space over $h:=\ell^2(\bn)$, as the map sending the algebraic
generators $a^{\dagger}_j$ and $a_j$ to the anti-symmetric creation and annihilation operators,
respectively \cite{BR2}. These operators will be explicitly recalled in Section \ref{AfdTralie},
where we return to the fermionic relations to solve a problem in duality that was raised in \cite{Dfer}.

\section{The Klein transformation}
\label{kjw}

We exhibit a simple but useful construction of a Jordan-Klein-Wigner (Klein for short) transformation. Such a (spatial) construction generalises the analogous one for the models living on the lattice $\bz^d$ and outlined in Section \ref{AfdZ2}, and includes those already present in literature ({\it e.g.} \cite{BR2, DHR1, DHR2, H}).

Consider a Hilbert space $\ch$, together with a self-adjoint unitary $\G\in\cb(\ch)$, {\it i.e.} $\G^2=\idd_\ch$, where $\idd_\ch$ is the identity operator on $\ch$.
The adjoint action $\g:=\ad_\G$ of $\G$ naturally equips
$\cb(\ch)$ with a $\bz_2$-grading.\footnote{Notice that $\g(\G)=\G\G\G^*=\G^3=\G$, hence $\G$ is even.} Notice that for any $a\in \cb(\ch)$
$$
\Gamma a_{+}=a_{+}\Gamma\,, \quad \Gamma a_{-}=-a_{-}\Gamma\,.
$$
We can define the {\it Klein transformation} $\k_\G$ on $\cb(\ch)$ induced by $\G$ as
$$
\k_\G(a_++a_-):=a_++\G a_-\,,\quad a_++a_-=a\in\cb(\ch)\,.
$$
Together with $\k_\G$, we can equally well define the map $\eta_\G$ by
\begin{equation}
\eta_\G(a_++a_-):=a_++\imath \G a_-\,,\quad a_++a_-=a\in\cb(\ch)\,.
\label{eta}
\end{equation}
Both $\k_\G$ and $\eta_\G$ are linear, invertible and identity preserving maps,
but $\eta_\G$ is a $*$-automorphism and $\k_\G$ is not unless $\G=\id_\ch$. Indeed, $\k^{-1}_\G=\k_\G$, and
\begin{equation}
\label{aa}
\eta^{-1}_\G(a_++a_-):=a_+-\imath \G a_-\,,\quad a_++a_-=a\in\cb(\ch)\,.
\end{equation}
Moreover,
$\k_\G(a^*)=\k_\G(\g(a))^*$, and $\k_\G(ab)-\k_\G(a)\k_\G(b)=2a_-b_-$ for any $a,b\in\cb(\ch)$.
We refer to $\eta_\G$ as the \emph{twist automorphism} of  $\cb(\ch)$ induced by $\G$.
Notice that in terms of the linear bijection
\begin{equation}
\widetilde{\varepsilon}:=\varepsilon _{1}+\imath \varepsilon _{-1}:\cb(\ch)\rightarrow \cb(\ch)
\label{epsilon}
\end{equation}
(whose inverse is given by $\widetilde{\varepsilon}^{-1}=\varepsilon_{1}-\imath
\varepsilon_{-1}$), we have that
\begin{equation}
\eta_{\Gamma }=\widetilde{\varepsilon}\circ\kappa_{\Gamma
}=\kappa_{\Gamma}\circ\widetilde{\varepsilon}\,, \label{kap&tau}
\end{equation}
which is relevant in relation to even maps in Section \ref{AfdDualiteit}.

The twist automorphism $\eta_\G$ of $\cb(\ch)$ is in fact the adjoint map of a unitary $K$ on $\ch$.
More in detail, let $P_{+}$ and $P_{-}$ the two projections onto the eigenspaces of $\G$, \emph{i.e.}
\begin{equation*}
P_{+}:=\frac{1}{2}(\idd_{\ch}+\Gamma )\text{ \ \ and \ \ }P_{-}:=\frac{1}{2}(\idd_{\ch}-\Gamma)\,,
\end{equation*}
and define the unitary
\begin{equation}
K:=P_{+}-\imath P_{-}\,.  \label{Kdef}\
\end{equation}
Observe that, for any $a\in \cb(\ch)$,
\begin{equation*}
K^{2}=\Gamma\,, \quad a_{+}K=Ka_{+}\,, \quad a_{-}K^{\ast
}=\imath Ka_{-}\,.
\end{equation*}
As a consequence,
\begin{equation}
\eta_{\Gamma}(a)=KaK^{\ast}  \label{kap&K}
\end{equation}
for any $a\in \cb(\ch)$.

The following definition gives a key concept for the paper, see \emph{e.g.} \cite{DHR1}, Section IV.
\begin{defin}
	\label{verwKom}
For a Hilbert space $\ch$, fix a self-adjoint unitary $\G\in\cb(\ch)$.
For any $S\subseteq \cb(\ch)$, the twisted commutant $S_\G^{\wr}$ of $S$ is defined as $S_\G^{\wr}:=\k_\G(S')$.
\end{defin}
\vskip.3cm
The twisted commutant $S_\G^{\wr}$ depends on the self-adjoint unitary $\G$. Sometimes, we omit to point out such a dependence if this causes no confusion.
We mention that the term ``graded commutant'' is also used for the twisted commutant, see \emph{e.g.} \cite{W}, Section 12.

For the remainder of this section we focus on the von Neumann algebraic set-up, which
forms the basis for duality in Sections \ref{AfdDualiteit} and \ref{AfdTralie}.
\begin{prop}
	\label{vKomForm}
Let $(\cam,\ch)$ be a von Neumann algebra and $\G\in \cb(\ch)$ a self-adjoint unitary. Then
$\cam_\G^{\wr}=\eta_\G(\cam')$, and therefore $\cam_\G^{\wr}$ is a von Neumann algebra.
\end{prop}
\begin{proof}
With $\eta=\eta_\G$ and $\cam^{\wr}=\cam_\G^{\wr}$, first we have
\begin{align*}
\cam^{\wr}=&\eps_1(\cam')\oplus\G\eps_{-1}(\cam')=\eps_1(\cam')\oplus\G\eps_{-1}(-\imath\cam')\\
=&\eps_1(\cam')\oplus\G\big(-\imath\eps_{-1}(\cam')\big)=\eps_1(\cam')\oplus\imath\G\eps_{-1}(\cam')\\
=&\eta(\cam')\,.
\end{align*}

The fact that $\cam^{\wr}$ is a von Neumann algebra easily follows if one checks $\eta(\cam')=\eta(\cam)'$. Indeed, suppose $x'\in\cam'$. For any $x\in\cam$
$$
\eta(x')\eta(x)=\eta(x'x)=\eta(xx')=\eta(x)\eta(x')\,,
$$
hence $\eta(\cam')\subset\eta(\cam)'$. Let now take $y\in\cb(\ch)$ such that $y\eta(x)-\eta(x)y=0$ for each $x\in\cam$, {\it i.e.} $y\in\eta(\cam)'$. Then
$$
\eta\big(\eta^{-1}(y)x-x\eta^{-1}(y))=0\iff \eta^{-1}(y)x=x\eta^{-1}(y)\,,\quad x\in\cam\,.
$$
Therefore, $\eta^{-1}(y)\in\cam'$,
that is $\eta(\cam)'\subset\eta(\cam')$, which combined with the previous computations leads to the assertion.
\end{proof}
As \eqref{aa} yields
\begin{equation}
\kappa _{\Gamma }(\mathcal{M})=\eta _{\Gamma}(\mathcal{M})=\eta _{\Gamma
}^{-1}(\mathcal{M})\,,  \label{kapVsInvkap}
\end{equation}
Proposition \ref{vKomForm} and \eqref{kapVsInvkap} give the following chain of equalities
\begin{equation}
\kappa _{\Gamma }(\mathcal{M}^{\prime })=\eta _{\Gamma }(\mathcal{M}^{\prime
})=\mathcal{M}_{\Gamma }^{\wr }=\eta _{\Gamma }(\mathcal{M})^{\prime }=\kappa _{\Gamma
}(\mathcal{M})^{\prime }\,.  \label{VerwKomItvTau}
\end{equation}
With an abuse of language, we still refer to the map $\eta_\G$
as a Klein transformation because it produces the same twisted commutant as $\k_\G$ and has
the added advantage of being a $*$-automorphism, which is crucial for the definition of the diagonal state in Section \ref{AfdDiag}.
\begin{prop}
\label{gam'}
Let  $(\cam, \ch)$ be a von Neumann algebra, and $\G\in\cb(\ch)$ a self-adjoint unitary.
Then
\begin{equation*}
\mathcal{M}_\G^{\wr \wr }:=(\mathcal{M}_\G^{\wr })_\G^{\wr }=\ad_{}\!_\G(\cam).
\end{equation*}
\end{prop}
\begin{proof}
With $\eta=\eta_\G$, we start by noticing that $\eta(\cam')=\eta(\cam)'$ ({\it cf.} \eqref{VerwKomItvTau}), and $\eta^2=\ad_\G$. Therefore, by Proposition \ref{vKomForm},
$$
(\mathcal{M}_\G^{\wr})_\G^\wr=\eta(\eta(\mathcal{M}^{\prime })^{\prime})=\eta(\eta(\mathcal{M})^{{\prime}{\prime}})
=\eta(\eta(\mathcal{M}))=\eta^2(\cam)=\ad_{}\!_\G(\cam)\,.
$$
\end{proof}
The following simple facts generalise the known ones on the standard vectors.
\begin{prop}
\label{gam''}
Let  $(\cam, \ch)$ be a von Neumann algebra, and $\G\in\cb(\ch)$ a self-adjoint unitary. For $\xi\in\ch$, and $K$ the unitary given in \eqref{Kdef}, we have
\begin{itemize}
\item[(i)] $\xi$ is cyclic for $\cam \iff K\xi$ is separating for $\cam^\wr_\G$,
\item[(ii)] $\xi$ is separating for $\cam \iff K\xi$ is cyclic for $\cam^\wr_\G$.
\end{itemize}
\end{prop}
\begin{proof}
We prove (i), (ii) being similar. Indeed, $\xi$ is cyclic for $\cam \iff \xi$ is separating for
$$
K\cam'=K\cam'K^*K=(K\cam'K^*)K=\cam^\wr_\G K\,,
$$
that is $K\xi$ is separating for $\cam^\wr_\G$, and the last equality follows from \eqref{kap&K}.
\end{proof}
A case of interest for our purposes will be when
$(\ga,\th)$ and $\f\in\cs(\ga)$ are a $\bz_2$-graded $C^*$-algebra and an even state, respectively.
Since $\f$ is $\th$-invariant, we can look at the GNS covariant representation
$\big(\ch_\f,\pi_\f,V_{\f,\th},\xi_\f\big)$. As $V_{\f.\th}$ is a self-adjoint unitary,
we can directly consider the $\bz_2$-grading induced on $\cb(\ch_\f)$ by
its adjoint action, making $\big(\cb(\ch_\f),\ad_{V_{\f,\th}}\big)$ an involutive $\bz_2$-graded algebra in a canonical way.

Indeed, we consider a von Neumann algebra $(\mathcal{M},\ch)$ with
a $\bz_2$-grading given by the involutive $*$-automorphism $\theta$, and a cyclic vector $\xi$. We suppose also that the normal
state
\begin{equation*}
\mu (b)=\left\langle b\xi ,\xi \right\rangle\,, \quad b\in \mathcal{M}
\end{equation*}
is even, \emph{i.e.} $\m\circ \theta=\m$. In this case we say $\theta$ is a $\bz_2$-grading
of $(\mathcal{M}, \m)$, and $\theta$ and $\mu$ can be extended to all of $\cb(\ch)$. Namely, the map
$$
\Gamma b\xi: =\theta (b)\xi\,, \quad b\in\cam
$$
uniquely extends to a self-adjoint unitary
operator on $\ch$, denoted by $\G$ with an abuse of notation.

The $\ast$-automorphism $\g$ of $\cb(\ch)$, defined as $\gamma:= \text{ad}_\G$ endows $\cb(\ch)$ with a $\bz_2$-grading structure, and indeed extends $\theta$ since $\theta(b)=\G b \G$ for any $b\in \mathcal{M}$.

For the functional $\mu$, it is simply extended as an even state by $\omega\in \cs(\cb(\ch))$, where
\begin{equation*}
\omega (a):=\left\langle a\xi ,\xi \right\rangle
\end{equation*}
for all $a\in \cb(\ch)$. Notice that
\begin{equation}
\om(\kappa _{\Gamma }(a)b)=\om(ab)   \label{tauEienskap}
\end{equation}
for all $a,b\in \cb(\ch)$, this property being useful in connection with duality in Section \ref{AfdDualiteit}.

Moreover, since $\G\xi=\xi$, $\xi$ is invariant for $K$. Thus (\ref{kap&K}) yields that $\eta _{\Gamma }(\mathcal{M})$ is a
von Neumann algebra with cyclic vector $\xi $, and  $\omega \circ \eta _{\Gamma }=\omega$.
\begin{rem}
\label{rem1}
We point out the following useful facts:

(i) $\mathcal{M}_\G^{\wr \wr}=\mathcal{M}$;

(ii) $\xi$ is separating for $\mathcal{M}_\G^{\wr}$;

(iii) if $\xi$ is separating for $\mathcal{M}$, then it is cyclic for $\mathcal{M}_\G^{\wr}$.

\bigskip

Indeed, (i) follows from Proposition \ref{gam'} as $\ad_{}\!_\G(\cam)=\cam$, whereas (ii) and (iii) are immediate consequence of Proposition \ref{gam''}, since in our case $K\xi=\xi$.
\end{rem}
As the adjoint action of $\G$ leaves $\cam$ globally stable, one has $\ad_{}\!_\G(\cam')=\cam'$. As a consequence, $\g':=\g\lceil_{\cam'}$ is a $\mathbb{Z}_{2}$-grading of $(\cam', \mu')$, where $\mu':=\om\lceil_{\cam'}$. In addition, $\ad_{}\!_\G(\cam_\G^{\wr})=\cam_\G^{\wr}$, and therefore $\gamma ^{\wr }:=\gamma \lceil_{\mathcal{M}_\G^{\wr }}$ is a $\mathbb{Z}_{2}$-grading of $(\mathcal{M}_\G^{\wr },\mu ^{\wr })$, where $\mu ^{\wr }:=\om\lceil_{M_\G^{\wr }}$.

Finally, we notice that we can take the right version of $\k_\G$ as the Klein transformation, that is
	\begin{equation}
	\kappa _{\Gamma }\circ \gamma =
	\gamma \circ \kappa _{\Gamma}:
	\cb(\ch)\rightarrow \cb(\ch):a\mapsto a_{+}+a_{-}\Gamma\,.  \label{linktau}
	\end{equation}
This forces us to get the right version of $\eta_{\Gamma}$,
given by $\eta_\G^{-1}$ in \eqref{aa}. However, because of (\ref{kapVsInvkap}) and Proposition \ref{vKomForm}, the twisted commutant is not affected by this choice, nor is there any effect on the content of our main results regarding duality collected in Theorem \ref{alfaGamma+} and Theorem \ref{Oplossing}.

\section{The (algebraic) Fermi tensor product}
\label{tensor}

Let $\ga$ and $\gb$ be two involutive $\bz_2$-graded algebras,
or equivalently $*$-superalgebras in the common language from Section \ref{AfdZ2}.
We now discuss the definition of the involutive $\bz_2$-graded tensor product
$\ga\, \circled{\text{{\tiny F}}}\, \gb$ between $\ga$ and $\gb$, which perhaps could be known to experts.

Concerning the linear structure, one notices that
$$
\ga\odot\gb=\dot{+}_{i,j\in\bz_2}(\ga_i\odot\gb_j)=\ga\otimes\gb\,,
$$
where on the r.h.s. we used the symbol ``$\otimes$" to recall
that the usual operations of taking adjoint ``$\dagger$" and product ``${\bf\cdot}$" are also defined on the l.h.s.\,.

For homogeneous elements $a\in\ga$, $b\in\gb$ and $i,j\in\bz_2$, we recall the following definitions
\begin{align}
\begin{split}
\label{ecfps}
&\eeps(a,b):=\left\{\!\!\!\begin{array}{ll}
                      -1 &\text{if}\,\, \partial(a)=\partial(b)=-1\,,\\
                     \,\,\,\,\,1 &\text{otherwise}\,.
                    \end{array}
                    \right.\\
&\eps(i,j):=\left\{\!\!\!\begin{array}{ll}
                      -1 &\text{if}\,\, i=j=-1\,,\\
                     \,\,\,\,\,1 &\text{otherwise}\,.
                    \end{array}
                    \right.
\end{split}
\end{align}
Consider the generic elements $x,y\in\ga\odot\gb$. We can write
\begin{align}
\begin{split}
\label{xy}
&x:=\oplus_{i,j\in\bz_2}x_{i,j}\in\oplus_{i,j\in\bz_2}(\ga_i\odot\gb_j)\,,\\
&y:=\oplus_{i,j\in\bz_2}y_{i,j}\in\oplus_{i,j\in\bz_2}(\ga_i\odot\gb_j)\,,
\end{split}
\end{align}
and we set
\begin{equation}
\label{prstc}
\begin{split}
x^*:=&\sum_{i,j\in\bz_2}\eps(i,j)x_{i,j}^\dagger\, ,\\
xy:=&\sum_{i,j,k,l\in\bz_2}\eps(j,k)x_{i,j}{\bf\cdot}y_{k,l}\,.
\end{split}
\end{equation}
Notice that for $x=a\odot b$ and $y=A\odot B$, where $a,A\in \ga$ and $b,B\in\gb$,
\begin{align*}
x^*=&\eeps(a,b)x^\dagger=\eeps(a,b)a^*\odot b^*\,,\\
xy=&\eeps(b,A)(a\otimes b){\bf\cdot}(A\otimes B)=\eeps(b,A)aA\otimes bB\,.
\end{align*}
\begin{prop}
The involution and product operations on $\ga\odot\gb$ given in \eqref{prstc} are well defined, and make the linear space $\ga\odot\gb$ an involutive algebra.
\end{prop}
\begin{proof}
It is almost immediate to show that the operations in \eqref{prstc} make $\ga\odot\gb$
an involutive algebra, provided they are well defined. For the latter property,
take $x$ and $y$ as in \eqref{xy}. If $x=0$ or $y=0$, then $x_{i,j}=0$ or $y_{i,j}=0$ for all $i,j\in\bz_2$, and therefore $x^*=0$ and $xy=0$.
\end{proof}
\begin{defin}
The tensor product $\ga\odot\gb$, endowed with the involution and product given in \eqref{prstc},
is called the (algebraic) Fermi tensor product of $\ga$ and $\gb$, and is denoted by
$\ga\, \circled{\rm{{\tiny F}}}\, \gb$.
\end{defin}
\vskip.3cm
For the involutive $\bz_2$-graded algebras $(\ga,\a)$ and $(\gb,\b)$,
their Fermi tensor product $\ga\, \circled{\rm{{\tiny F}}}\, \gb$ is naturally equipped with a structure of involutive
$\bz_2$-graded algebra, by putting
\begin{equation}
\label{picm}
\begin{split}
\big(\ga\, \circled{\rm{{\tiny F}}}\, \gb\big)_+:=&\big(\ga_+\odot\gb_+\big)\dot{+}\big(\ga_{-}\odot\gb_{-}\big)\,,\\
\big(\ga\, \circled{\rm{{\tiny F}}}\, \gb\big)_-:=&\big(\ga_+\odot\gb_{-}\big)\dot{+}\big(\ga_{-}\odot\gb_{+}\big)\,.
\end{split}
\end{equation}
In this situation, such a grading is induced by the involutive automorphism $\th=\a\, \circled{\rm{{\tiny F}}}\, \b$ given on the elementary tensors by
\begin{equation}
\label{aupf}
(\a\, \circled{\rm{{\tiny F}}}\, \b)(a\, \circled{\rm{{\tiny F}}}\, b):=\a(a)\, \circled{\rm{{\tiny F}}}\, \b(b)\,,\quad a\in\ga\,,\,\, b\in\gb\,.
\end{equation}

\section{The product state}
\label{AfdProdToest}
In the present section, for the involutive $\bz_2$-graded algebras $(\ga,\a)$ and $(\gb,\b)$, and for given $\om\in\cs(\ga)$ and $\f\in\cs(\gb)$,
we study whether the product functional of $\om$ and $\f$ is well defined and
positive on the Fermi tensor product $\ga\, \circled{\rm{{\tiny F}}}\, \gb$.

If needed, we suppose without loosing generality that $\ga$ and $\gb$ are
unital with units $\idd_\ga$ and $\idd_\gb$, respectively. The general situation can be generalised by adding the identities, or considering approximate identities on $\ga$ and $\gb$ as well.

Although a detailed study of the product functional relative to concrete systems based
on the CAR algebra can be found in \cite{AM1, AM2}, here we provide the analysis for the abstract Fermi systems we are dealing with.

To this aim, we first denote by $\cs(\ga)_+$ the convex set of the even states on $\ga$,
and observe that the linear structure of
the algebraic tensor product involutive algebra $\ga\, \circled{\rm{{\tiny F}}}\, \gb$ is indeed the same as $\ga\otimes \gb$ and $\ga\odot\gb$. Namely,
$$
\ga\otimes \gb=\ga\odot\gb=\ga\, \circled{\rm{{\tiny F}}}\, \gb
$$
as vector spaces.
For $\om\in\cs(\ga)$ and $\f\in\cs(\gb)$, the product functional $\om\times\f$ on $\ga\, \circled{\rm{{\tiny F}}}\,\gb$ is defined as usual by
\begin{equation}
\label{finprc}
\om\times\f\bigg(\sum_{j=1}^n a_j\, \circled{\rm{{\tiny F}}}\,b_j\bigg):=\sum_{j=1}^n\om(a_j)\f(b_j),\quad \sum_{j=1}^n a_j\, \circled{\rm{{\tiny F}}}\,b_j\in \ga\, \circled{\rm{{\tiny F}}}\,\gb\,.
\end{equation}
It is well known
that the functional given in \eqref{finprc} is well defined on $\ga\odot\gb$, and
therefore on $\ga\, \circled{\rm{{\tiny F}}}\, \gb$. In addition, even though it coincides
with the state $\psi_{\om,\f}$ on $\ga\otimes \gb$, in general it is not positive on the involutive algebra $\ga\, \circled{\rm{{\tiny F}}}\, \gb$.

In what follows, we show that the positivity criterion for product states established in
\cite{AM1} also holds for the most general case treated in the present paper.
\begin{prop}
\label{prst}
Let $\om\in\cs(\ga)$, $\f\in\cs(\gb)$ and suppose $\om$  or $\f$ is even. Then $\om\times\f$ is positive on $\ga\, \circled{\rm{{\tiny F}}}\, \gb$.
In addition, for any $x,y\in \ga\, \circled{\rm{{\tiny F}}}\, \gb$ one has
\begin{equation}
\begin{split}
\label{prodstar}
|(\om\times\f)(x)|\leq&(\om\times\f)(x^*x)^{1/2}\,,\\
(\om\times\f)(x^*y^*yx)\leq&C_y(\om\times\f)(x^*x)\,,
\end{split}
\end{equation}
where $C_y$ is a positive constant depending on $y$.
\end{prop}
\begin{proof}
We can suppose without loss of generality that $\om$ is even, \emph{i.e.} $\om\circ\a=\om$, the other case being similar.
Concerning the positivity of the product state, we compute for homogeneous $a,A\in\ga$ and $b,B\in\gb$
\begin{align*}
(\om\times\f)\big[(a\, \circled{\rm{{\tiny F}}}\, b)^*(A\, \circled{\rm{{\tiny F}}}\, B)\big]
=&(\om\times\f)\big[(a^*\, \circled{\rm{{\tiny F}}}\, b_+^*)(A\, \circled{\rm{{\tiny F}}}\, B)\big]\\
+&(\om\times\f)\big[(\a(a^*)\, \circled{\rm{{\tiny F}}}\, b_-^*)(A\, \circled{\rm{{\tiny F}}}\, B)\big]\\
=&(\om\times\f)\big(a^*A\, \circled{\rm{{\tiny F}}}\, b_+^*B\big)\\
+&(\om\times\f)\big(\a(a^*A)\, \circled{\rm{{\tiny F}}}\, b_-^*B\big)\\
=&\om(a^*A)\f(b^*B)\\
=&\psi_{\om,\f}\big[(a\otimes b)^\dagger{\bf\cdot}(A\otimes B)\big]\,.
\end{align*}
Therefore, for general
$$
c:=\sum_{i=1}^n a_i\, \circled{\rm{{\tiny F}}}\,  b_i\in\ga\, \circled{\rm{{\tiny F}}}\, \gb=\ga\otimes\gb
$$
it follows that
$$
(\om\times\f)(c^*c)=\sum_{i,j=1}^n\om(a_i^*a_j)\f(b_i^*b_j)=\psi_{\om,\f}(c^\dagger{\bf\cdot} c)\geq0\,.
$$
Concerning the last assertion, we consider the GNS representation
$\big(\ch_{\psi_{\om,\f}},\pi_{\psi_{\om,\f}},\xi_{\psi_{\om,\f}}\big)$ of
$\psi_{\om,\f}\in\cs(\ga\otimes_{\rm min}\gb)$. Then for $x,y\in \ga\, \circled{\rm{{\tiny F}}}\, \gb$,
\begin{align*}
|(\om\times\f)(x)|=&\big|\big\langle\pi_{\psi_{\om,\f}}(x)\xi_{\psi_{\om,\f}},\xi_{\psi_{\om,\f}}\big\rangle\big|\leq
\big\|\pi_{\psi_{\om,\f}}(x)\xi_{\psi_{\om,\f}}\big\|_{\ch_{\psi_{\om,\f}}}\\
=&\psi_{\om,\f}(x^\dagger{\bf\cdot}x)^{1/2}=(\om\times\f)(x^*x)^{1/2}\,.
\end{align*}
Analogously,
\begin{align*}
&(\om\times\f)(x^*y^*yx)=\psi_{\om,\f}(x^\dagger{\rm\cdot}y^\dagger{\rm\cdot}y{\rm\cdot}x)\\
=&\big\langle\pi_{\psi_{\om,\f}}(y)\pi_{\psi_{\om,\f}}(x)\xi_{\psi_{\om,\f}},\pi_{\psi_{\om,\f}}(y)\pi_{\psi_{\om,\f}}(x)\xi_{\psi_{\om,\f}}\big\rangle\\
=&\big\|\pi_{\psi_{\om,\f}}(y)\pi_{\psi_{\om,\f}}(x)\xi_{\psi_{\om,\f}}\big\|_{\ch_{\psi_{\om,\f}}}^2\\
\leq&\big\|\pi_{\psi_{\om,\f}}(y)\big\|_{\cb(\ch_{\psi_{\om,\f}})}^2\big\|\pi_{\psi_{\om,\f}}(x)\xi_{\psi_{\om,\f}}\big\|_{\ch_{\psi_{\om,\f}}}^2 \\
=&\big\|\pi_{\psi_{\om,\f}}(y)\big\|_{\cb(\ch_{\psi_{\om,\f}})}^2(\om\times\f)(x^*x)\,,
\end{align*}
and \eqref{prodstar} follows after taking
$C_y=\big\|\pi_{\psi_{\om,\f}}(y)\big\|_{\cb(\ch_{\psi_{\om,\f}})}^2\,.$
\end{proof}

\section{The Fermi tensor product $C^*$-algebra}
\label{fermicstar}
The present section is devoted to complete the Fermi algebraic product
$\ga\, \circled{\rm{{\tiny F}}}\,\gb$ of two $\bz_2$-graded $C^*$-algebras $\ga$ and $\gb$ w.r.t. the standard maximal norm, obtaining
$\ga\, \circled{\rm{{\tiny F}}}_{\rm max}\gb$.

On $\ga\, \circled{\rm{{\tiny F}}}\,\gb$, we define the seminorm
\begin{equation}
\label{norm}
[]c[]_{\rm max}:=\sup\{\|\pi(c)\|\mid\pi\,\,\text{is a representation}\},\,\, c\in\ga\, \circled{\rm{{\tiny F}}}\,\gb\,.
\footnote{By $[]\,\,\,[]_{\rm max}$, we are denoting the maximal norm \eqref{norm} on the $\bz_2$-graded algebra
$\ga\, \circled{\text{{\tiny F}}}\, \gb$ to distinguish that from the maximal $C^*$-cross norm$\|\,\,\,\|_{\rm max}$
on the usual tensor product $\ga\otimes\gb$.}
\end{equation}
\begin{theorem}
The seminorm in \eqref{norm} is indeed a $C^*$-norm on $\ga\, \circled{\rm{{\tiny F}}}\,\gb$.
\end{theorem}
\begin{proof}
We start by noticing that for any $c\in \ga\, \circled{\rm{{\tiny F}}}\,\gb$ \eqref{norm}
satisfies $[]c^*c[]_{\rm max}=[]c[]^2_{\rm max}$, that is it defines a $C^*$-seminorm.

Fix now an element in $c\in \ga\, \circled{\rm{{\tiny F}}}\,\gb$ of the form $c=\sum_{i=1}^n\l_ia_i\, \circled{\rm{{\tiny F}}}\,\ b_i$
with $\l_i\in\mathbb{C}$, $a_i\in\ga$, $b_i\in\gb$, and $\|a_i\|,\,\|b_i\|\leq1$. By reasoning as in Lemma 3.3 in \cite{Fid}, it follows that
$$
\|\pi(c)^*\pi(c)\|\leq n\sum_{i=1}^n|\l_i|^2<+\infty.
$$
Therefore $[]c[]_{\rm max}$ is finite, since the inequality above is independent of the representation $\pi$.

Suppose that the $\bz_2$-grading on $\ga$ is induced by $\a\in\aut(\ga)$, and denote $\eps:=\frac{1}{2}(\id_{\ga}+\a)$
the conditional expectation of $\ga$ onto $\ga_+=\eps(\ga)$, the $C^*$-subalgebra of even elements.
In order to check that $c\neq 0\Rightarrow []c[]_{\rm max}>0$ for each such
$$
c\in\ga\, \circled{\rm{{\tiny F}}}\,\gb=\ga\otimes\gb\subset\ga\otimes_{\rm min}\gb\,,
$$
we fix states $\om\in\cs(\ga)$, $\f\in\cs(\gb)$ such that $\psi_{\om\circ\eps,\f}(c^\dagger{\bf\cdot}c)>0$.
These states exist because the set of product states separates the points of $\ga\otimes_{\rm min}\gb$, see {\it e.g.} \cite{T},
Theorem IV.4.9 (iii). By \eqref{prodstar} and Theorem \ref{gnsrcfp}, we can take the GNS representation
$\big(\ch_{\om\circ\eps\times\f},\pi_{\om\circ\eps\times\f},\xi_{\om\circ\eps\times\f}\big)$ and compute
\begin{align*}
[]c[]^2_{\rm max}\geq&\|\pi_{\om\circ\eps\times\f}(c)\|^2\geq\|\pi_{\om\circ\eps\times\f}(c)\xi_{\om\circ\eps\times\f}\|^2\\
=&\|\pi_{\psi_{\om\circ\eps,\f}}(c)\xi_{\psi_{\om\circ\eps,\f}}\|^2=\psi_{\om\circ\eps,\f}(c^\dagger{\bf\cdot}c)>0\,.
\end{align*}
\end{proof}
\begin{defin}
The completion of $\ga\, \circled{\rm{{\tiny F}}}\,\gb$ w.r.t. the norm in \eqref{norm} is denoted
by $\ga\, \circled{\rm{{\tiny F}}}_{\rm max}\gb$ and simply called the Fermi $C^*$-tensor product between the $\bz_2$-graded $C^*$-algebras $\ga$ and $\gb$.
\end{defin}
\vskip.3cm
\begin{rem}
Notice that $\ga\, \circled{\rm{{\tiny F}}}_{\rm max}\gb$ is naturally endowed with the following structure of $\bz_2$-graded algebra
$$
\ga\, \circled{\rm{{\tiny F}}}_{\rm max}\gb=(\ga\, \circled{\rm{{\tiny F}}}_{\rm max}\gb)_{+}\oplus(\ga\, \circled{\rm{{\tiny F}}}_{\rm max}\gb)_{-}
$$
where
$$
(\ga\, \circled{\rm{{\tiny F}}}_{\rm max}\gb)_{\pm}=\overline{(\ga\, \circled{\rm{{\tiny F}}}\,\gb)_{\pm}}
$$
for $(\ga\, \circled{\rm{{\tiny F}}}\,\gb)_{\pm}$ given in \eqref{picm}. Such a grading is induced by (the extension of)
the involutive automorphism in \eqref{aupf}.
\end{rem}
For the sake of completeness, we report the following statement, which is nothing else than the Fermi version of Proposition IV.4.7 of \cite{T}.
\begin{theorem}
\label{tk}
Let $(\ga_i,\a_i)$, $i=1,2$ be $\bz_2$-graded unital $C^*$-algebras, and $\gb$ an arbitrary unital $C^*$-algebra. If the unital $*$-homomorphisms
$$
\pi_i:\ga_i\to\gb,\quad i=1,2\,,
$$
satisfy
\begin{equation}
\label{com}
\pi_1(a_1)\pi_2(a_2)=\epsilon(a_1,a_2) \pi_2(a_2)\pi_1(a_1)\,,
\end{equation}
where $\eeps$ is given in \eqref{ecfps}, and $a_1\in \ga_1$, $a_2\in \ga_2$ are homogeneous elements,
then there exists a unique $*$-homomorphism $\pi$ of $\ga_1\, \circled{\rm{{\tiny F}}}_{\rm max}\ga_2$ into $\gb$ such that
$$
\pi(a_1\, \circled{\rm{{\tiny F}}}\,a_2)=\pi_1(a_1)\pi_2(a_2),\quad a_i\in\ga_i,\,\,i=1,2.
$$
Moreover, $\pi\big(\ga_1\, \circled{\rm{{\tiny F}}}_{\rm max}\ga_2\big)$ is the $C^*$-subalgebra of $\gb$ generated by $\pi_1(\ga_1)$ and $\pi_2(\ga_2)$.

If in addition $\gb$ is a $\bz_2$-graded $C^*$-algebra whose grading is generated by the involutive $\b\in\aut(\gb)$ satisfying
$$
\pi_i\circ \a_i=\b\circ \pi_i\,,\quad i=1,2\,,
$$
then $\pi\circ(\a_1\, \circled{\rm{{\tiny F}}}\,\a_2)=\b\circ\pi$.
 \end{theorem}
\begin{proof}
The map $(a_1,a_2)\in \ga_1\times \ga_2 \mapsto \pi_1(a_1)\pi_2(a_2)\in\gb$ is a bilinear form. Therefore, by universal property of the tensor product
$\ga_1\odot \ga_2=\ga_1\, \circled{\rm{{\tiny F}}}\,\ga_2$, there is a unique linear map $\pi_o: \ga_1\, \circled{\rm{{\tiny F}}}\,\ga_2\rightarrow \gb$ such that
$$
\pi_o(a_1\, \circled{\rm{{\tiny F}}}\,a_2)=\pi_1(a_1)\pi_2(a_2),\quad a_i\in\ga_i,\,\,i=1,2\,.
$$
We now check that $\pi_o$ is a $*$-homomorphism. Indeed, for $c:=\sum_{i=1}^m x_i\, \circled{\rm{{\tiny F}}}\, y_i$ and
$d:=\sum_{j=1}^n r_j\, \circled{\rm{{\tiny F}}}\, s_j$ in
$\ga_1\, \circled{\rm{{\tiny F}}}\,\ga_2$ with the $x_i$, $y_i$, $r_j$ and $s_j$ homogeneous, we get by \eqref{prstc} and \eqref{com},
\begin{align*}
\pi_o(c^*)=&\sum_{i=1}^m\pi_o\big(\eeps(x_i,y_i)\pi_o(x_i^*\, \circled{\rm{{\tiny F}}}\, y_i^*\big)=\sum_{i=1}^m\eeps(x_i,y_i)\pi_1(x_i^*)\pi_2(y_i^*)\\
=&\sum_{i=1}^m\big(\epsilon(x_i,y_i)\pi_2(y_i)\pi_1(x_i)\big)^*
=\sum_{i=1}^m\big(\pi_1(x_i)\pi_2(y_i)\big)^*\\
=&\bigg(\sum_{i=1}^m\pi_1(x_i)\pi_2(y_i)\bigg)^*=\pi_o(c)^*\,,
\end{align*}
and
\begin{align*}
\pi_o(cd)=&\sum_{i=1}^m\sum_{j=1}^n\pi_o\big((x_i\, \circled{\rm{{\tiny F}}}\, y_i)(r_j\, \circled{\rm{{\tiny F}}}\, s_j)\big)\\
=&\sum_{i=1}^m\sum_{j=1}^n\eeps(r_j,y_i)\pi_o\big((x_i r_j\, \circled{\rm{{\tiny F}}}\, y_i s_j)\big)\\
=&\sum_{i=1}^m\sum_{j=1}^n\pi_1(x_i)\big(\eeps(r_j,y_i)\pi_1(r_j)\pi_2(y_i)\big)\pi_2(s_j)\\
=&\bigg(\sum_{i=1}^m\pi_1(x_i)\pi_2(y_i)\bigg)\bigg(\sum_{j=1}^n\pi_1(r_j)\big)\pi_2(s_j)\bigg)\\
=&\pi_o(c)\pi_o(d)\,.
\end{align*}
$\pi_o$ is defined on the dense involutive subalgebra $\ga_1\, \circled{\rm{{\tiny F}}}\, \ga_2$ of
$\ga_1\, \circled{\rm{{\tiny F}}}_{\rm max}\ga_2$, then arguing as in \cite{CrF2}, Proposition 4.1 and Proposition 4.3, one finds
it uniquely extends to a bounded map, denoted by $\pi$, on the whole $\ga_1\, \circled{\rm{{\tiny F}}}_{\rm max}\ga_2$
which is indeed a $*$-homomorphism. As an immediate consequence, $\pi\big(\ga_1\, \circled{\rm{{\tiny F}}}_{\rm max}\ga_2\big)$
is generated as a $C^*$-algebra by $\pi_1(\ga_1)$ and $\pi_2(\ga_2)$.

Concerning the last assertion, it is enough to prove this on the dense set of generators of the
form $c=\sum_{i=1}^nx_i\, \circled{\rm{{\tiny F}}}\, y_i$, obtaining by \eqref{aupf}
\begin{align*}
\pi\big((\a_1\, \circled{\rm{{\tiny F}}}\, \a_2&)(c)\big)=\sum_{i=1}^n\pi\big(\a_1(x_i)\, \circled{\rm{{\tiny F}}}\,\a_2(y_i)\big)
=\sum_{i=1}^n\pi_1(\a_1(x_i))\pi_2(\a_2(y_i))\\
=&\sum_{i=1}^n\b(\pi_1(x_i))\b(\pi_2(y_i))=\b\bigg(\sum_{i=1}^n\pi_1(x_i)\pi_2(y_i)\bigg)
=\b(\pi(c))\,.
\end{align*}
\end{proof}
We end the present section by briefly discussing the following fact, for the sake of completeness.
Indeed, define
\begin{equation*}
[]c[]_{\rm min}:=\sup\{\|\pi_{\f\times\psi}(c)\|\mid \f\in\cs(\ga)_+,\,\,\psi\in\cs(\gb)_+\},\quad c\in\ga\, \circled{\rm{{\tiny F}}}\,\gb\,.
\end{equation*}
Obviously, $[]c[]_{\rm min}\leq[]c[]_{\rm max}$. Since the set of the even product states considered above separates the points of $\ga\, \circled{\rm{{\tiny F}}}\,\gb$
({\it cf.} \cite{T},
Theorem IV.4.9 (iii)), $[]c[]_{\rm min}$ is actually a norm, and
$$
\overline{\ga\, \circled{\rm{{\tiny F}}}\,\gb}^{[]\,\,[]_{\rm min}}\sim\overline{\big(\oplus_{\{\f,\psi\mid \f\in\cs(\ga)_+,\,\psi\in\cs(\gb)_+\}}\pi_{\f\times\psi}\big)\big(\ga\, \circled{\rm{{\tiny F}}}\,\gb\big)}\,.
$$
This norm can be viewed as the analogue of the minimal $C^*$-cross norm in the Fermi situation,
see {\it e.g.} \cite{T}, Definition IV.4.8, for the case involving the usual tensor product.

In the sequel we never use $\ga\, \circled{\rm{{\tiny F}}}\,\gb$ equipped with $[]c[]_{\rm min}$. So
we do not to pursue this analysis further, postponing a more exhaustive treatment to somewhere else.

\section{The GNS representation of the product state}
\label{AfdProdGNS}

The present section is devoted to describe in some detail the GNS representation of
(the extension of) the product state on the Fermi $C^*$-tensor product between two $\bz_2$-graded $C^*$-algebras.

As before, $\eps:\ga:\to\ga_+$ is the conditional expectation of the $\bz_2$-graded $C^*$-algebra $(\ga, \th)$ onto its even part.

Since $[\ga_+,\gb]=0$ (in $\ga\, \circled{\rm{{\tiny F}}}_{\rm max}\gb$) with $\gb$ also a $\bz_2$-graded $C^*$-algebra, we note that
$$
\ga_+\, \circled{\rm{{\tiny F}}}\,\gb=\ga_+\otimes\gb\,\,\text{and}\,\,\ga_+\, \circled{\rm{{\tiny F}}}_{\rm max}\gb=\ga_+\otimes_{\rm max}\gb\,.
$$
Define now the linear map
$$
E_o:\ga\, \circled{\rm{{\tiny F}}}\,\gb\to\ga_+\, \circled{\rm{{\tiny F}}}\,\gb\subset\ga_+\, \circled{\rm{{\tiny F}}}_{\rm max}\gb
=\ga_+\otimes_{\rm max}\gb
$$
given, on the homogeneous elements $a\in\ga, b\in\gb$, by
$$
E_o(a\, \circled{\rm{{\tiny F}}}\,b):=
\eps(a)\, \circled{\rm{{\tiny F}}}\,b=\eps(a)\otimes b\,.
$$
It is well defined by the universal property of tensor products, and in addition the following result holds true.
\begin{lem}
\label{leco}
The linear map $E_o$ is completely positive, and therefore it extends to a bounded linear map
$$
E:\ga\, \circled{\rm{{\tiny F}}}_{\rm max}\gb\rightarrow\ga_+\, \circled{\rm{{\tiny F}}}_{\rm max}\gb=\ga_+\otimes_{\rm max}\gb\,,
$$
which is indeed a conditional expectation, preserving the identity if $\ga$ and $\gb$ are unital.
\end{lem}
\begin{proof}
It is easy to see that $E_o$ leaves the elements of $\ga_+\, \circled{\rm{{\tiny F}}}\, \gb$ invariant
and preserves the identity, provided $\ga$ and $\gb$ are unital. Moreover, it is a real map. Indeed, for $a\in\ga$ and $b\in\gb$ \eqref{prstc} gives
\begin{align*}
&E_o(a\, \circled{\rm{{\tiny F}}}\,b)^*
=(\eps(a)\, \circled{\rm{{\tiny F}}}\, b)^*
=\eps(a)^*\, \circled{\rm{{\tiny F}}}\, b^*
\\
=&\eps(a^*)\, \circled{\rm{{\tiny F}}}\, b^* =E_o(a^*\, \circled{\rm{{\tiny F}}}\, b^*)=E_o\big((a\, \circled{\rm{{\tiny F}}}\, b)^*\big)\,.
\end{align*}
In particular, this implies that the bimodule property for $E_o$ follows as soon as
we show it is right $\ga_+\, \circled{\rm{{\tiny F}}}\, \gb$-linear. To this aim, consider
$$
x:=\sum_{i=1}^m a_i\, \circled{\rm{{\tiny F}}}\, b_i\in \ga\, \circled{\rm{{\tiny F}}}\, \gb\,, \quad
y:=\sum_{j=1}^n c_j\, \circled{\rm{{\tiny F}}}\, d_j\in \ga_+\, \circled{\rm{{\tiny F}}}\, \gb\,,
$$
where the elementary factors in the tensors can be taken homogeneous. Since the $c_j$ belong to $\ga_+$,
$$
\eeps(b_i,c_j)=1\,,\quad i=1,\dots,m\,,\,\,\,j=1,\dots,n\,.
$$
Therefore, as $\eps$ is a $\ga_+$-bimodule map, we have
\begin{align*}
E_o(xy)=&\sum_{i=1}^m\sum_{j=1}^n E_o(a_ic_j\, \circled{\rm{{\tiny F}}}\, b_id_j)
=\sum_{i=1}^m\sum_{j=1}^n \eps(a_i)c_j\, \circled{\rm{{\tiny F}}}\,  b_id_j \\
=&\sum_{i=1}^m\sum_{j=1}^n \big(\eps(a_i)\, \circled{\rm{{\tiny F}}}\, b_i\big)\big(c_j\, \circled{\rm{{\tiny F}}}\, d_j\big)\\
=&\bigg(\sum_{i=1}^m\eps(a_i)\, \circled{\rm{{\tiny F}}}\, b_i\bigg)\bigg(\sum_{j=1}^n c_j\, \circled{\rm{{\tiny F}}}\, d_j\bigg)\\
=&E_o(x)y\,.
\end{align*}
We now show that $E_o$ is completely positive. By reasoning as in Proposition 9.3 of \cite{St}, it is enough to verify that it is positive.
To this goal, first take  homogeneous elements $a,c\in\ga$ and $b,d\in\gb$. Then
$$
(a\, \circled{\rm{{\tiny F}}}\, b)^*(c\, \circled{\rm{{\tiny F}}}\, d)=\eeps(a,b)\eeps(c,b)(a^*c\, \circled{\rm{{\tiny F}}}\, b^*d)\,.
$$
We note that if $\partial(a)\partial(c)=-1$ then $\eps(a^*c)=0$, and if  $\partial(a)\partial(c)=1$ then $\eeps(a,b)=\eeps(c,b)$. Therefore,
$$
E_o\big((a\, \circled{\rm{{\tiny F}}}\, b)^*(c\, \circled{\rm{{\tiny F}}}\, d)\big)=\eps(a^*c)\, \circled{\rm{{\tiny F}}}\, b^*d\,.
$$
Thus in all situations, for $x=\sum_{i=1}^n a_i\, \circled{\rm{{\tiny F}}}\,b_i\in \ga\, \circled{\rm{{\tiny F}}}\,\gb$ we obtain
\begin{align*}
E_o(x^*x)=
&E_o\bigg(\bigg(\sum_{i=1}^n a_i\, \circled{\rm{{\tiny F}}}\, b_i\bigg)^*
\bigg(\sum_{j=1}^n a_j\, \circled{\rm{{\tiny F}}}\,b_j\bigg)\bigg)\\
=&\sum_{i,j=1}^n \eps(a^*_ia_j)\, \circled{\rm{{\tiny F}}}\,b^*_ib_j\\
=&\sum_{k=1}^m \sum_{i,j=1}^n c_i(k)^* c_j(k)\, \circled{\rm{{\tiny F}}}\,b^*_ib_j\\
=&\sum_{k=1}^m\bigg(\sum_{i=1}^nc_i(k)\, \circled{\rm{{\tiny F}}}\,b_i\bigg)^*\bigg(\sum_{i=1}^nc_i(k)\, \circled{\rm{{\tiny F}}}\,b_i\bigg)\geq 0\,,
\end{align*}
 where we used the fact that, for some integer $m$, the positive matrix $[\eps(a^*_ia_j)]_{i,j=1}^n\in\bm_n(\ga_+)$
 can always be written as a sum of $m$ positive matrices
$\big\{[c_i(k)^*c_j(k)]_{i.j=1}^n\big\}_{k=1}^m\subset\bm_n(\ga_+)$ (see {\it e.g.} \cite{T},  Lemma IV.3.1).

The complete positivity of $E_o$ allows us to argue as in \cite{CrF2}, Proposition 4.1 and conclude
that it uniquely extends to a bounded linear map $E$ on the whole
$\ga\, \circled{\rm{{\tiny F}}}_{\rm max}\gb$, which turns out to be a conditional expectation onto
$\ga_+\, \circled{\rm{{\tiny F}}}_{\rm max}\gb=\ga_+\otimes_{\rm max} \gb$.
\end{proof}

Let $\F:\ga\to\gb$ be a completely positive map between the $C^*$-algebras $\ga$ and $\gb$,
together with a positive linear functional $\f$ on $\gb$. Obviously,
$\f\circ\F$ is a positive linear functional on $\ga$. Let $(\ch_\f,\pi_\f,\xi_\f)$ and
$(\ch_{\f\circ\F},\pi_{\f\circ\F},\xi_{\f\circ\F})$ be the GNS representations of $\f$ and
$\f\circ\F$, respectively. Consider the Stinespring dilation $(\ch,\pi, V)$ of $\pi_\f\circ\F:\ga\to\cb(\ch_\f)$, see {\it e.g.} \cite{T}, Theorem IV.3.6.
Let $P\in\pi(\ga)'$ be the self-adjoint projection onto the cyclic subspace $\pi(\ga)V\xi_\f\subset\ch$.
\begin{lem}
\label{leco1}
The GNS representation  $(\ch_{\f\circ\F},\pi_{\f\circ\F},\xi_{\f\circ\F})$ of the positive functional $\f\circ\F$ is $(P\ch,P\pi,V\xi_\f)$.
\end{lem}
\begin{proof}
In order to prove the assertion, it is enough to check that, if $a\in\ga$ then $\langle P\pi(a)V\xi_\f,V\xi_\f\rangle=\f\circ\F(a)$. Indeed,
\begin{align*}
\langle P\pi(a)V\xi_\f,V\xi_\f\rangle=&\langle\pi(a)V\xi_\f,V\xi_\f\rangle=\langle V^*\pi(a)V\xi_\f,\xi_\f\rangle\\
=&\langle\pi_\f(\F(a))\xi_\f,\xi_\f\rangle=\f(\F(a))\,.
\end{align*}
\end{proof}
\begin{prop}
Let $\ga$ and $\gb$ be $\bz_2$-graded $C^*$-algebras and $\om\in\cs(\ga)_+$, $\f\in\cs(\gb)$. Consider the product state $\psi_{\om,\f}$ on
$\ga_+\otimes_{\rm max}\gb=\ga_+\, \circled{\rm{{\tiny F}}}\,_{\rm max}\gb$.

The GNS representation of the product state
$\om\times\f\in\cs\big(\ga\, \circled{\rm{{\tiny F}}}\,_{\rm max}\gb\big)$ is given by $\big(P\ch,P\pi,V\xi_{\psi_{\om,\f}}\big)$, where:
\begin{itemize}
\item[(i)] $\big(\ch_{\psi_{\om,\f}},\pi_{\psi_{\om,\f}},\xi_{\psi_{\om,\f}}\big)$ is the GNS representation of $\psi_{\om,\f}$;
\item[(ii)] $(\ch,\pi,V)$ is the Stinespring dilation of the completely positive map
$\pi_{\psi_{\om,\f}}\circ E:\ga\, \circled{\rm{{\tiny F}}}\,_{\rm max}\gb\to\cb\big(\ch_{\psi_{\om,\f}}\big)$,
with $E$ the conditional expectation given in Lemma \ref{leco};
\item[(iii)] $P\in\pi\big(\ga\, \circled{\rm{{\tiny F}}}\,_{\rm max}\gb\big)'$ is the self-adjoint projection onto the cyclic subspace
$\pi\big(\ga\, \circled{\rm{{\tiny F}}}\,_{\rm max}\gb\big)V\xi_{\psi_{\om,\f}}\subset\ch$.
\end{itemize}
\end{prop}
\begin{proof}
As $\psi_{\om,\f}\circ E=\om\times \f$, the proof directly follows from the previous lemmata.
\end{proof}

\section{The diagonal state}
\label{AfdDiag}

In the present section we investigate the Fermi counterpart of the diagonal state defined in \cite{Fid}, which plays a crucial role to define the (fermionic) detailed balance.

We start by recalling some basic facts about the {\it the opposite algebra} $\ga^\circ$ of a
fixed involutive algebra $\ga$. Indeed, $\ga^\circ$ is $\ga$ as a linear space,
with the product $x\circ y:=yx$ and involution
$(x^\circ)^*:=(x^*)^\circ$.

For any map $T:\ga\rightarrow X$ from $\ga$ to the point-set $X$, the {\it opposite map} $T^\circ:\ga^\circ\rightarrow X$ is simply defined as
$$
T^\circ(a^\circ):=T(a)\,,\quad a\in\ga\,.
$$
We note that if $\F:\ga\rightarrow\gb$ is a (completely) positive linear map between involutive algebras, then $\F^\circ$ is also (completely) positive, whereas if $\F$ is a
$*$-homomorphism, then so is $\F^\circ$.

Correspondingly,
if $(\ga,\th)$ is $\bz_2$-graded, then $(\ga^\circ,\th^\circ)$
is also a $\bz_2$-graded involutive algebra in an obvious manner and, for an even linear functional $f:\ga\rightarrow\bc$, the opposite functional $f^\circ$ is also even.

In order to define the diagonal state, we start with an even state $\f\in\cs(\ga)_+$ of the
$\bz_2$-graded $C^*$-algebra $(\ga,\th)$ with central support $s_\f\in Z(\ga^{**})$
in the bidual.\footnote{A state $\f\in\cs(\ga)$ on a $C^*$-algebra $\ga$ has central support if and only
if whenever it is considered as a state on the bidual $W^*$-algebra $\ga^{**}$, the cyclic vector
$\xi_\f\in\ch_\f$ is also separating for $\pi_\f(\ga)''$, see \cite{NSZ}, p.15.}
Since $\f$ is invariant under $\th$ and has central support in the bidual, we can look
at its covariant GNS representation $\big(\ch_\f,\pi_\f,V_{\f,\th},\xi_\f\big)$ and
consider the Tomita conjugation $J_\f$, acting on $\ch_\f$, associated to $\f$. We report
the following well-known facts:
\begin{equation}
\label{ccomf}
\begin{split}
&J_\f\pi_\f(\ga)J_\f=\pi_\f(\ga)^\prime\,, \\
&V_{\f,\th}\pi(x)=\pi(\th(x))V_{\f,\th}\,, \quad x\in\ga \\
&J_\f V_{\f,\th}=V_{\f,\th}J_\f\,.
\end{split}
\end{equation}
Take the $\bz_2$-graded $*$-algebra $\ga\, \circled{\rm{{\tiny F}}}\,\ga^\circ$, together with the functional
$$
\d^o_\f:\ga\, \circled{\rm{{\tiny F}}}\,\ga^\circ\to\bc\,,
$$
and the linear map
$$
\pi_{\d_\f}^o:\ga\, \circled{\rm{{\tiny F}}}\,\ga^\circ\to\cb(\ch_\f)\,,
$$
defined on the generators by
\begin{equation}
\label{fidiag}
\begin{split}
\d^o_\f(a\, \circled{\rm{{\tiny F}}}\,b^\circ):=&\langle\pi_\f(a)\eta_{V_{\f,\th}}\big(J_\f\pi_\f(b^*)J_\f\big)\xi_\f,\xi_\f\rangle\,,\\
\pi_{\d_\f}^o(a\, \circled{\rm{{\tiny F}}}\,b^\circ):=&\pi_\f(a)\eta_{V_{\f,\th}}\big(J_\f\pi_\f(b^*)J_\f\big)\,,\,\,\,a,b\in\ga\,.
\end{split}
\end{equation}
The above maps are obviously well defined by the universal property of
the algebraic tensor product $\ga\odot\ga^\circ=\ga\, \circled{\rm{{\tiny F}}}\,\ga^\circ$. Their main properties are summarised in the following
\begin{thm}
\label{gprdig}
The linear maps $\d^o_\f$ and $\pi_{\d_\f}^o$ respectively extend to a state $\d_\f$,
and to a representation $\pi_{\d_\f}$ of $\ga\, \circled{\rm{{\tiny F}}}_{\rm max}\ga^\circ$ satisfying:
\begin{itemize}
\item[(i)] the GNS representation of $\d_\f$ is $(\ch_\f,\pi_{\d_\f},\xi_\f)$;
\item[(ii)] $\pi_{\d_\f}\big(\ga\, \circled{\rm{{\tiny F}}}_{\rm max}\ga^\circ\big)''$ is the von Neuman algebra
$\pi_\f(\ga)''\bigvee\pi_\f(\ga)^\wr_{V_{\f,\th}}$ generated by $\pi_\f(\ga)$ and its twisted commutant $\pi_\f(\ga)^\wr_{V_{\f,\th}}$.
\end{itemize}
\end{thm}
\begin{proof}
The theorem is proved once we show that $\pi_{\d_\f}^o$ extends to a $*$-representation
of $\ga\, \circled{\rm{{\tiny F}}}_{\rm max}\ga^\circ$, which is equivalent to show
that $\pi_{\d_\f}^o$ is a $*$-representation of $\ga\, \circled{\rm{{\tiny F}}}\,\ga^\circ$.
The remaining parts are left to the reader.

Fix $a,b,c,d\in\ga$ homogeneous, and take into account \eqref{ccomf}.
For the adjoint, if $\partial(b)=1$ we easily get
$$
\pi_{\d_\f}^o\big((a\, \circled{\rm{{\tiny F}}}\,b^\circ)^*\big)=\pi_{\d_\f}^o(a\, \circled{\rm{{\tiny F}}}\,b^\circ)^*
$$
Concerning the remaining cases where $\partial(b)=-1$, we get
\begin{align*}
&\pi_{\d_\f}^o(a\, \circled{\rm{{\tiny F}}}\,b^\circ)^*=\big(\pi_\f(a)(\imath V_{\f,\th}J_\f\pi_\f(b^*)J_\f)\big)^*\\
=&-\imath J_\f\pi_\f(b)J_\f V_{\f,\th} \pi_\f(a^*)
=\imath V_{\f,\th}J_\f\pi_\f(b)J_\f\pi_\f(a^*)\\
=&\imath V_{\f,\th}\pi_\f(a^*)J_\f\pi_\f(b)J_\f
=\partial(a)\imath\pi_\f(a^*)V_{\f,\th}J_\f\pi_\f(b)J_\f\\
=&\partial(a)\pi_\f(a^*)\big(\imath V_{\f,\th}J_\f\pi_\f(b)J_\f\big)
=\pi_{\d_\f}^o\big((a\, \circled{\rm{{\tiny F}}}\,b^\circ)^*\big)\,.
\end{align*}
Concerning the product, if $\partial(b)=1=\partial(d)$ or if $\partial(b)=1$ and $\partial(d)=-1$, we again easily get
\begin{align*}
\pi_{\d_\f}^o\big((a\, \circled{\rm{{\tiny F}}}\,b^\circ)(c\, \circled{\rm{{\tiny F}}}\,d^\circ)\big)
=&\pi_{\d_\f}^o(ac\otimes b^\circ d^\circ)=\pi_{\d_\f}^o(ac\otimes db)\\
=&\pi_\f(a)\pi_\f(c)
J_\f\pi_\f(b^*)J_\f\eta_{V_{\f,\th}}\big(J_\f\pi_\f(d^*)J_\f\big)\\
=&\pi_\f(a)J_\f\pi_\f(b^*)J_\f\pi_\f(c)
\eta_{V_{\f,\th}}\big(J_\f\pi_\f(d^*)J_\f\big) \\
=&\pi_{\d_\f}^o(a\, \circled{\rm{{\tiny F}}}\,b^\circ)\pi_{\d_\f}^o(c\, \circled{\rm{{\tiny F}}}\,d^\circ)\,.
\end{align*}
Let now  $\partial(b)=-1$ and $\partial(d)=1$, we get
\begin{align*}
\pi_{\d_\f}^o(a\, \circled{\rm{{\tiny F}}}\,b^\circ)\pi_{\d_\f}^o(c\, \circled{\rm{{\tiny F}}}\,d^\circ)
=&\pi_\f(a)(\imath V_{\f,\th}J_\f\pi_\f(b^*)J_\f)\pi_\f(c)J_\f\pi_\f(d^*)J_\f\\
=&\partial(c)\pi_\f(ac)(\imath V_{\f,\th}J_\f\pi_\f(b^*d^*)J_\f)\\
=&\pi_{\d_\f}^o\big((a\, \circled{\rm{{\tiny F}}}\,b^\circ)(c\, \circled{\rm{{\tiny F}}}\,d^\circ)\big)\,.
\end{align*}
Finally, with $\partial(b)=-1=\partial(d)$ we get
\begin{align*}
\pi_{\d_\f}^o(a\, \circled{\rm{{\tiny F}}}\,b^\circ)\pi_{\d_\f}^o(c\, \circled{\rm{{\tiny F}}}\,
d^\circ)=&\pi_\f(a)(\imath V_{\f,\th}J_\f\pi_\f(b^*)J_\f)\pi_\f(c)(\imath V_{\f,\th}J_\f\pi_\f(d^*)J_\f)\\
=&\partial(c)\pi_\f(ac)J_\f\pi_\f(b^*d^*)J_\f \\
=&\pi_{\d_\f}^o\big((a\, \circled{\rm{{\tiny F}}}\,b^\circ)(c\, \circled{\rm{{\tiny F}}}\,d^\circ)\big)\,.
\end{align*}
\end{proof}
The state $\d_\f$ above defined is called the {\it diagonal state} associated with $\f$.

Suppose for simplicity that $\ga$ is unital, and look at the marginals of $\d_\f$. Since
$$
\d_\f\lceil_{\ga\, \circled{\rm{{\tiny F}}}_{\max}\,\idd_\ga^\circ}=\f\,,\quad \d_\f\lceil_{\idd_\ga\, \circled{\rm{{\tiny F}}}_{\max}\,\ga^\circ}=\f^\circ\,,
$$
$\d_\f$ can be considered the diagonal state associated to the product state
$\f\times\f^\circ\in\cs\big(\ga\, \circled{\rm{{\tiny F}}}_{\rm max}\ga^\circ\big)$,
in analogy to the classical and the usual tensor product cases, see {\it e.g.} \cite {D, Fid}.

We also note that the diagonal state $\d_\f$ is in general not normal ({\it i.e.} not ``absolutely continuous") w.r.t. the product state $\f\times\f^\circ$, unless
$\pi_{\f\times\f^\circ}\big(\ga\, \circled{\rm{{\tiny F}}}_{\rm max}\ga^\circ\big)''$ is atomic.\footnote{A $W^*$-algebra $\gam$ is said to be {\it atomic} if it is generated by the set of its minimal projections. It turns out to be equivalent to the fact that $\gam$ is a direct sum of type $\ty{I}$ factors.}

To end the present section, we briefly discuss the case when $\gam$ is a $\bz_2$-graded $W^*$-algebra, and $\f\in\cs(\gam)$ a normal even state.

For a pair of $\bz_2$-graded $W^*$-algebras
$(\gam,\a)$ and $(\gn,\b)$, we can define on $\gam\, \circled{\rm{{\tiny F}}}\,\gn$ the
{\it maximal binormal norm} as follows. As before, if $c\in\gam\, \circled{\rm{{\tiny F}}}\,\gn$, we put
\begin{equation*}
[]c[]^{\rm bin}_{\rm max}:=\sup\{\|\pi(c)\|\mid\pi\,\text{is a representation s.t.}\, \pi\lceil_\gam,\,\pi\lceil_\gn\,\text{are normal}\}\,.
\end{equation*}
Obviously,
$$
[]c[]^{\rm bin}_{\rm max}\leq []c[]_{\rm max}\,.
$$
The $C^*$-algebra $\gam\, \circled{\rm{{\tiny F}}}^{\rm bin}_{\rm max}\gn$
is nothing else than the completion of $\gam\, \circled{\rm{{\tiny F}}}\,\gn$ w.r.t. the above norm.\footnote{For the concept of binormal tensor product $\ga\otimes_{\rm min}^{\rm bin}\gb$ of $W^*$-algebras, that is the completion of the algebraic tensor product $\ga\otimes\gb$ of the $W^*$-algebras $\ga$ and $\gb$ w.r.t. the minimal binormal $C^*$-cross norm, see {\it e.g.} \cite{J}.}

For any normal faithful even state $\f$ on the $\bz_2$-graded $W^*$-algebra $(\gam,\th)$,
the diagonal state $\d_\f^o$ extends to a state, denoted with an abuse of notation also with $\d_\f$, on
$\gam\, \circled{\rm{{\tiny F}}}^{\rm bin}_{\rm max}\gam^\circ$ as well. Its GNS representation is described as in Theorem \ref{gprdig}, and
$$
\pi_{\d_\f}\big(\gam\, \circled{\rm{{\tiny F}}}_{\rm max}\gam^\circ\big)''=\pi_{\d_\f}\big(\gam\, \circled{\rm{{\tiny F}}}^{\rm bin}_{\rm max}\gam^\circ\big)''=\pi_\f(\gam)\bigvee \pi_\f(\gam)^\wr_{V_{\f,\th}}\,.
$$
We leave the details to the reader.

\section{Duality}

\label{AfdDualiteit}

This section sets up a duality theory for positive linear maps between von Neumann algebras via the twisted commutants.

Consider two von Neumann algebras $(\mathcal{M}, \ch_{\mu})$ and $(\mathcal{N}, \ch_{\nu})$
with cyclic vectors $\xi_\m$ and $\xi_\n$ respectively, and states \begin{equation*}
\mu (a)=\left\langle a\xi _{\mu },\xi_{\mu }\right\rangle  \quad \nu (b)=\left\langle b\xi_{\nu },\xi_{\nu }\right\rangle\,,
\end{equation*}
for all $a\in \mathcal{M}$ and $b\in \mathcal{N}$. Assume that the algebras are $\bz_2$-graded, \emph{i.e.}
we take $(\mathcal{M},\th_\mu)$ and $(\mathcal{N},\th_\nu)$, and that $\m$ and $\n$ are also even.
Let $\Gamma _{\mu}\in B(\ch_{\mu})$ and $\Gamma _{\nu}\in B(\ch_{\nu})$ be the self-adjoint unitary operators which allow to extend $\th_{\mu}$ and $\th _{\nu}$ to the whole $B(\ch_{\mu})$ and $B(\ch_{\nu})$ by $\g_\mu:=\ad_{\Gamma _{\mu}}$ and $\g_\n:=\ad_{\Gamma _{\nu}}$, respectively.
For the Klein transformations and twisted $*$-automorphisms we use the shorthand notations
$\kappa _{\mu}$, $\eta _{\mu}$ and $\kappa_\nu$, $\eta_\nu$, respectively.

Recall that for any positive linear map $\Psi :\mathcal{M}\rightarrow \mathcal{N}$ such that $\nu \circ \Psi =\mu$,
its dual $\Psi ^{\prime }:\mathcal{N}^{\prime }\rightarrow \mathcal{M}^{\prime }$ is defined via
\begin{equation}
\label{duall}
\left\langle \Psi ^{\prime }(b')a\xi_{\mu },\xi_{\mu }\right\rangle
=\left\langle b'\Psi (a)\xi_{\nu },\xi_{\nu }\right\rangle
\end{equation}
for all $a\in \mathcal{M}$ and $b'\in \mathcal{N}^{\prime}$. The reader is referred to \cite{AC},
Proposition 3.1 and \cite{DS2}, Theorem 2.5 for more details and properties. In what follows we remove the subscripts for the
twisted commutants of $\mathcal{M}$ and $\mathcal{N}$, since no confusion can arise. For $\mu ^{\wr }$ and $\nu ^{\wr}$ defined as in Section \ref{kjw}, one has the following
\begin{proposition}
	\label{verwrBehoud}
For any positive linear map $\Psi:\mathcal{M}\rightarrow \mathcal{N}$ such that $\nu\circ \Psi=\m$,
there exists a unique unital $\Psi ^{\wr}:\mathcal{N}^{\wr}\rightarrow \mathcal{M}^{\wr}$ satisfying
\begin{equation}
\left\langle \Psi ^{\wr}(b^{\wr})a\xi_{\mu },\xi_{\mu }\right\rangle
=\left\langle b^{\wr}\Psi (a)\xi_{\nu },\xi_{\nu }\right\rangle
\label{verwrDuaal}
\end{equation}
for all $a\in \mathcal{M}$ and $b^{\wr}\in \mathcal{N}^{\wr}$. In addition, if $\Psi$ is
unit preserving then $\mu ^{\wr }\circ \Psi ^{\wr }=\nu ^{\wr}$,
and $\Psi^{\wr }$ is faithful.
\end{proposition}
\begin{proof}
Indeed, let us take
\begin{equation}
\label{verwrDuDef}
\Psi ^{\wr }:=\kappa_{\mu }\circ \Psi^{\prime }\circ \kappa_{\nu}\lceil_{\mathcal{N}^{\wr}}\,.
\end{equation}
Since $\kappa_{\nu }^{-1}(\mathcal{N}^{\prime })=\kappa _{\nu }(\mathcal{N}^{\prime })=\mathcal{N}^{\wr }$, the last equality coming from \eqref{VerwKomItvTau},
$\Psi ^{\wr }: \mathcal{N}^{\wr}\rightarrow \mathcal{M}^{\wr}$. Thus, condition \eqref{verwrDuaal} follows
from (\ref{tauEienskap}) and \eqref{duall}. Namely, for all $a\in \mathcal{M}$ and $b^{\wr}\in \mathcal{N}^{\wr}$
	\begin{align*}
	\left\langle \Psi ^{\wr }(b^{\wr})a\xi_{\mu },\xi_{\mu }\right\rangle &
	=\left\langle \Psi ^{\prime }(\kappa _{\nu }(b^{\wr}))a\xi_{\mu },\xi
	_{\mu }\right\rangle \\
&=\left\langle \kappa _{\nu }(b^{\wr})\Psi (a)\xi_{\nu
	},\xi_{\nu }\right\rangle  \\
	& =\left\langle b^{\wr}\Psi (a)\xi_{\nu },\xi_{\nu }\right\rangle\,.
	\end{align*}
 Note that \eqref{verwrDuaal} uniquely determines $\Psi ^{\wr }$ as $\xi_{\mu }$ is cyclic for $\mathcal{M}$,
 and separating for $\mathcal{M}^{\wr }$ by Remark \ref{rem1}. Moreover, Theorem 2.5 in \cite{DS2} gives that $\Psi^\prime$ is unital, and then \eqref{verwrDuDef} entails $\Psi ^{\wr}$ is identity preserving. Assuming that $\Psi(\idd_{\mathcal{M}})=\idd_{\mathcal{N}}$, we get from \eqref{verwrDuaal}
	\begin{equation*}
	(\mu ^{\wr }\circ \Psi ^{\wr })(b^{\wr})=\left\langle \Psi ^{\wr }(b^{\wr})\xi
	_{\mu },\xi_{\mu }\right\rangle =\left\langle b^{\wr} \Psi(\idd_{\mathcal{M}})\xi_{\nu
	},\xi_{\nu }\right\rangle =\nu ^{\wr }(b^{\wr})
	\end{equation*}
	for all $b^{\wr}\in \mathcal{N}^{\wr }$. As a consequence, if $\Psi ^{\wr }((b^{\wr})^{\ast}b^{\wr})=0$ one has $\nu ^{\wr }((b^{\wr})^{\ast }b^{\wr})=0$. Therefore $b^{\wr}=0$, since $\xi _{\nu}$ is separating for $\mathcal{N}^{\wr }$.
\end{proof}
We call $\Psi ^{\wr}$ the \emph{twisted dual} of $\Psi$. Achieving positivity of the twisted dual needs a further assumption on $\Psi$, namely that it has to be even,
as shown in the next result.
\begin{theorem}
	\label{alfaGamma+}
Assume that $\Psi $ is grading-equivariant, \emph{i.e.} $\Psi \circ \g _{\mu }=\g_{\nu }\circ \Psi$, and $\nu\circ \Psi=\m$. Then $\Psi ^{\wr }$ as defined in (\ref{verwrDuDef}) is
	positive, even, with unit norm, and is given by
	\begin{equation*}
	\Psi ^{\wr }=\eta _{\mu}\circ \Psi ^{\prime }\circ\eta_{\nu}^{-1}\lceil_{\mathcal{N}^{\wr}}=
\eta_{\m}^{-1}\circ \Psi ^{\prime }\circ \eta _{\nu}\lceil_{\mathcal{N}^{\wr}}\,.
	\end{equation*}
	In addition,
	
	(a) if $\Psi$ is $n$-positive for some $n\in\bn$, then so is $\Psi^{\wr}$;
	
	(b) if $\Psi$ is completely positive, then so is $\Psi^{\wr}$;
	
	(c) if $\Psi $ is unit preserving, while $\xi_{\mu }$ and $\xi_{\nu }$ are
	separating for $\mathcal{M}$ and $\mathcal{N}$ respectively, then
	\begin{equation*}
	\Psi ^{\wr \wr }:=(\Psi ^{\wr })^{\wr }:
	\mathcal{M}\rightarrow \mathcal{N}
	\end{equation*}
	is well-defined and
	\begin{equation*}
	\Psi ^{\wr \wr }=\Psi .
	\end{equation*}
\end{theorem}
\begin{proof}
Recall that $\G_\mu\xi_\mu=\xi_\mu$, $\G_\nu\xi_\nu=\xi_\nu$, and $\g_{\mu }(\mathcal{M}^{\prime })=\mathcal{M}^{\prime }$, $\g_{\n}(\mathcal{N}^{\prime })=\mathcal{N}^{\prime }$.
 Then, since $\g_\nu\circ \Psi=\Psi\circ \g_\mu$, for $a\in \mathcal{M}$ and $b'\in \mathcal{N}^{\prime }$
we get from \eqref{duall} that
\begin{align*}
	\left\langle \Psi ^{\prime }(\g _{\nu }(b'))a\xi_{\mu },\xi
	_{\mu }\right\rangle & =\left\langle \Gamma _{\nu
	}b'\Gamma _{\nu
	}\Psi(a)\xi_{\nu },\xi_{\nu }\right\rangle
=\left\langle b'\Gamma _{\nu
	}\Psi (a)\Gamma _{\nu }\xi_{\nu },\xi_{\nu }\right\rangle \\
	&=\left\langle b'\g _{\nu }(\Psi(a))\xi_{\nu },\xi_{\nu
	}\right\rangle =\left\langle b'\Psi(\g _{\mu }(a))\xi_{\nu
	},\xi_{\nu }\right\rangle \\
	&=\left\langle \g _{\mu }(\Psi ^{\prime }(b'))a\xi_{\mu },\xi_{\mu }\right\rangle\,.
	\end{align*}
As $\xi_{\mu }$ is cyclic for $\mathcal{M}$, and therefore separating for $\mathcal{M}^{\prime}$, it follows that
	\begin{equation*}
	\Psi ^{\prime }\circ \g _{\nu }=\g _{\mu }\circ \Psi ^{\prime }\,.
	\end{equation*}
Since both $\kappa_{\nu}$ and $\kappa _{\m}$ are even, from \eqref{verwrDuDef} one finds $\Psi ^{\wr }$ is grading-equivariant, \emph{i.e.}
$$
\Psi ^{\wr }\circ \g_{\nu }=\g_{\mu
	}\circ \Psi ^{\wr }\,.
$$
It also follows that
	\begin{equation}
\label{3circ}
	\Psi ^{\prime }\circ \widetilde{\varepsilon} _{\nu }=\widetilde{\varepsilon} _{\mu }\circ \Psi
	^{\prime }\,,
	\end{equation}
	for $\widetilde{\varepsilon} _{\mu }$ and $\widetilde{\varepsilon} _{\nu }$ defined as in
	(\ref{epsilon}). As $k_\nu=k_\nu^{-1}$, by \eqref{kap&tau} and \eqref{3circ} one has
	\begin{align*}
	\Psi ^{\wr }& =\widetilde{\varepsilon} _{\mu }\circ \widetilde{\varepsilon} _{\mu }^{-1}\circ
	\kappa _{\mu}\circ \Psi ^{\prime }\circ \kappa _{\nu}\lceil_{\mathcal{N}^{\wr}} \\
    &=\widetilde{\varepsilon} _{\mu
	}\circ \kappa _{\mu}\circ \Psi ^{\prime }\circ \widetilde{\varepsilon} _{\nu
	}^{-1}\circ \kappa _{\nu }^{-1}\lceil_{\mathcal{N}^{\wr}} \\
	& =\eta_{\mu}\circ \Psi ^{\prime }\circ \eta _{\nu }^{-1}\lceil_{\mathcal{N}^{\wr}}\,.
	\end{align*}
	Similarly, $\Psi ^{\wr }=\eta _{\mu }^{-1}\circ \Psi ^{\prime }\circ
	\eta _{\nu }\lceil_{\mathcal{N}^{\wr}}$. As $\eta _{\mu }$ and $\eta _{\nu }$ are $\ast $-automorphisms, the
	positivity of $\Psi ^{\wr }$ follows from that of $\Psi ^{\prime }$.
	Likewise, $
	\left\| \Psi ^{\prime }\right\| =1$ implies $\left\| \Psi ^{\wr }\right\| =1$.
	
	Arguing as above, the $n$-positivity and complete positivity of $\Psi^{\wr}$ follow from the corresponding properties of $\Psi^{\prime}$,
	which in turn are implied from that of $\Psi$, by \cite{AC}, Proposition 3.1. This proves (a) and (b).
	
	Lastly, for (c) since $\Psi (\idd_{\mathcal{M}})=\idd_{\mathcal{N}}$, by Proposition \ref{verwrBehoud} we have $\mu ^{\wr }\circ \Psi
	^{\wr }=\nu ^{\wr }$. Furthermore, Remark \ref{rem1} ensures that
$\xi_{\mu }$ and $\xi _{\nu }$ are cyclic for the von Neumann algebras $
	\mathcal{M}^{\wr }$ and $\mathcal{N}^{\wr }$, respectively. Reasoning as in the case of $\Psi^{\wr}$, one sees that
	\begin{equation*}
	\Psi ^{\wr \wr }=\kappa _{\nu }\circ (\Psi ^{\wr})^{\prime }\circ
	\kappa _{\mu }\lceil_{\mathcal{M}}
	\end{equation*}
is indeed a map from $\mathcal{M}$ to $\mathcal{N}$, as a consequence of (\ref{VerwKomItvTau}). Moreover, it is well-defined as we assumed
	that $\Psi $ is positive, and $\nu \circ \Psi =\mu$.
	
	Since $\Psi ^{\wr }$ and $\Psi$ are positive, for any $a\in \mathcal{M}$ and $b^{\wr}\in
	\mathcal{N}^{\wr }$, by \eqref{verwrDuaal} one finds
	\begin{align*}
	\left\langle \Psi ^{\wr \wr }(a)b^{\wr}\xi_{\nu },\xi_{\nu
	}\right\rangle & =\left\langle a\Psi ^{\wr }(b^{\wr})\xi _{\mu },\xi_{\mu
	}\right\rangle =\overline{\left\langle \Psi ^{\wr }((b^{\wr})^{\ast })a^{\ast
		}\xi_{\mu },\xi_{\mu }\right\rangle } \\
	& =\overline{\left\langle (b^{\wr})^{\ast }\Psi (a^{\ast })\xi_{\nu },\xi_{\nu }\right\rangle }=\left\langle \Psi (a)b^{\wr}\xi_{\nu },\xi_{\nu
	}\right\rangle .
	\end{align*}
	Therefore $\Psi ^{\wr \wr }(a)=\Psi (a)$, since $\xi_{\nu }$ is separating for $\mathcal{N}$, and cyclic for $\mathcal{N}^{\wr }$.
\end{proof}
Similar to the proof above, one sees that when $\Psi ^{\wr}$ is positive, (\ref{verwrDuaal}) can also be expressed as
\begin{equation}
\left\langle a\Psi ^{\wr}(b^{\wr})\xi_{\mu },\xi_{\mu }\right\rangle
=\left\langle \Psi (a)b^{\wr}\xi_{\nu },\xi_{\nu }\right\rangle
\label{linkerduaal}
\end{equation}
for all $a\in \mathcal{M}$ and $b^{\wr}\in \mathcal{N}^{\wr }$. Thus, when $\Psi$ is grading-equivariant a duality in terms of the
	twisted commutants can be given in precise analogy to duality in terms of the commutants.
\begin{remark}
We notice that for $\lambda=\m,\nu$, our convention  to use $\kappa_\l$ rather than
	$\kappa_\l\circ\g_\l$ (see \eqref{linktau}) as the Klein transformation
has no effect on Theorem \ref{alfaGamma+}. Indeed, in the latter case one would exploit the form (\ref{linkerduaal}) for the twisted
	dual of $\Psi$, rather than (\ref{verwrDuaal}). It can be easily checked
that in the proof of Proposition \ref{verwrBehoud}, instead of (\ref{tauEienskap}) we would then use the property
	\begin{equation*}
	\left\langle a\kappa_\l(\g_\l(b))\xi ,\xi \right\rangle
	=\left\langle ab\xi ,\xi \right\rangle
	\end{equation*}
	for $a,b\in \cb(\ch)$, thus obtaining the same twisted dual and the same
results in Theorem \ref{alfaGamma+} for both the Klein transformation conventions.
\end{remark}

\section{The lattice}

\label{AfdTralie}

In this section, as an application of the theory developed above, we solve,
even in a more general form, a problem from \cite{Dfer} which we briefly describe for the convenience of the reader.

Let $h$ be a finite dimensional or separable infinite dimensional (complex)
Hilbert space, describing the space for a single fermion particle.
Denote the resulting Fermi (or anti-symmetric) Fock space as $\ch$, with $\langle \cdot,\cdot\rangle$ as inner product,
linear in the first variable, and $f_{\varnothing }$ as
the vacuum vector. The linear space $\ch$ is nothing else than the Fock representation of $\carf(L)$
introduced in Section \ref{AfdZ2}, when $L$ is finite (\emph{i.e.} $h=\bc^{|L|}$) or $L\sim\bn$ (\emph{i.e.} $h=\ell^2(\bn)$),
and we refer the reader to \cite{BR2}, Section 5.2 for a treatment of the matter,
and for the description of the basic operators on it as well. The latter ones are indeed the creation and annihilation operators denoted by
$a^{\dagger}(x)$, and $a(x)$ for $x\in h$, respectively. They have unit norm on $\ch$, are mutually adjoint, and satisfy the anticommutation relations
\begin{equation}
\begin{split}
\label{c0}
&\{a(x),a(y)\}=0 \\
&\{a^{\dagger}(x),a(y)\}=\left\langle x,y\right\rangle \idd_{\ch}
\end{split}
\end{equation}
for all $x,y\in h$.

The lattice $L$ indexes an orthonormal basis for
$h$, say $(e_{l})_{l\in L}$, and we use the notation
\begin{equation*}
a_{l}:=a(e_{l})\,, \quad a_{l}^{\dagger}:=a^{\dagger}(e_{l})\,, \quad l\in L\,.
\end{equation*}
After recalling that $a_l f_\varnothing =0$ for all $l\in L$, we also write
\begin{equation}
f_{(l_{1},\ldots,l_{n})}:=a_{l_{1}}^{\dagger}...a_{l_{n}}^{\dagger}f_{\varnothing}
\label{f}
\end{equation}
for all $l_{1},\ldots,l_{n}\in L$, and $n\geq 1$.

For any subset $I$ of $L$, let $D_{I}$ be a set of finite sequences
$(l_{1},...,l_{n})$ in $I$, for $n=0,1,2,3,...$, with $l_{j}\neq l_{k}$ when
$j\neq k$, such that each finite subset of $I$ corresponds to exactly one
element of $D_{I}$. The empty subset of $I$ corresponds to the sequence with
$n=0$, which is denoted by $\emptyset\in D_{I}$. Note that $D_{I}$ is countable
when $I$ is infinite.
For $s=(l_{1},\ldots,l_{n})$ and $t=(k_{1},\ldots,k_{m})$ in $D_I$ such that $l_i\neq k_h$
for any $i,h$, one denotes $st:=(l_{1},\ldots,l_{n},k_{1},\ldots,k_{m})\in D_I$. Here,
for $t=\emptyset$ or $s=\emptyset$ one has $st=s$ or $st=t$, respectively.
Note that $(f_s)_{s\in D_L}$ is an orthonormal basis for $\ch$, where $f_{s}$ for $s=(l_{1},\ldots,l_{n})\in D_{L}$ is given by (\ref{f}).

Let $(\mathcal{A}(I),\ch)$ be the von Neumann algebra generated by $\{a_{l}:l\in I\}$,
and consider a set of probability outcomes $(p_{s})_{s\in D_{I}}$, \emph{i.e.} $p_{s}\geq 0$, and $\sum_{s\in D_{I}}p_{s}=1$.
In addition, let $\iota:I\rightarrow \iota(I)\subseteq L$ be a bijection such that $I\cap \iota(I)= \emptyset$.

In what follows we consider two states. First, $\tr(\rho _{I}\,\cdot)\in \cs(\mathcal{A}(I))$ such that $\rho _{I}$ is the diagonal density matrix
\begin{equation}
\rho _{I}=\sum_{s\in D_{I}}p_{s}f_{s}\Join f_{s}\,,  \label{roI}
\end{equation}
where $x\Join y\in \cb(\ch)$ is defined as
\begin{equation*}
(x\Join y)z:=\left\langle z,y\right\rangle x
\end{equation*}
for all $x,y,z\in\ch$. To obtain the second state, we introduce
the so-called fermionic entangled vector
\begin{equation}
\zeta:=\sum_{s\in D_{I}}p_{s}^{1/2}f_{s\iota (s)}\in\ch  \label{diag}
\end{equation}
where $\iota (s)=(\iota (l_{1}),...,\iota (l_{n}))$ if $s=(l_{1},...,l_{n})$.
The fermionic entangled pure state $\varphi\in \cs(\mathcal{A}(I\cup \iota (I)))$ is defined by
\begin{equation*}
\varphi (a):=\left\langle a\zeta,\zeta \right\rangle
\end{equation*}
for all $a\in \mathcal{A}(I\cup \iota (I))$. It is straightforward to show that
$\varphi\lceil_{\mathcal{A}(I)}=\tr(\rho _{I}\,\cdot)$,
and $\varphi\lceil_{\mathcal{A}(I)}=\tr(\rho _{\iota(I)}\,\cdot)$. Moreover, $\varphi $ leads to the following bilinear form
\begin{equation*}
\mathcal{A}(I)\times \mathcal{A}(\iota (I))\ni (a,b)\mapsto B_{\varphi}(a,b):=\varphi (ab)\in \bc\,.
\end{equation*}
Assuming that $I$ is finite and $p_{s}>0$ for all $s\in D_{I}$, in
\cite{Dfer}, Theorem 7.3, it was shown that any linear map
\begin{equation*}
\Psi :\mathcal{A}(I)\rightarrow \mathcal{A}(I).
\end{equation*}
has a unique (necessarily linear) \emph{fermionic dual} map
\begin{equation*}
\Psi ^{\varphi }:\mathcal{A}(\iota (I))\rightarrow \mathcal{A}(\iota (I))
\end{equation*}
such that for all $a\in \mathcal{A}(I)$ and $b\in \mathcal{A}(\iota (I))$
\begin{equation*}
B_{\varphi }(\Psi (a),b)=B_{\varphi }(a,\Psi ^{\varphi }(b))\,.
\end{equation*}
In the same paper, after showing that $B_{\varphi}(a,b)$ need not be positive when $a$ and $b$ are,
the question was raised whether there are assumptions ensuring that $\Psi^{\varphi}$ inherits positivity,
$n$-positivity or completely positivity from $\Psi$. Here, we solve this problem even in the more general
case of countable $I$.

The problem is not affected when $L$ is replaced by $I \cup  \iota(I)$,
since in this case we restrict $\ch$ to the Hilbert space spanned by $\{e_{l} \mid l\in I\cup \iota(I)\}$,
say $\ch_{I\cup  \iota(I)}$, which gives a faithful representation of
$\mathcal{A}(I\cup \iota(I))$. Thus, without loss of generality,
we can work in terms of
\begin{equation}
\iota :I\rightarrow L\backslash I\,.  \label{komplBij}
\end{equation}
To reach our goal we apply a result stated in \cite{BJL}, where
the twisted duality of the CAR algebra is expressed in terms of the so-called self-dual
approach due to Araki (see \cite{A1} and \cite{A2}, as well as the review
\cite{A1987}), rather than in the usual terms of creation and annihilation
operators. In the next lines we briefly remind the reader
how to connect the two formulations.

Consider the Hilbert space
\begin{equation*}
h^{2}:=h\oplus h
\end{equation*}
with orthonormal basis given by the vectors $(e_{k},e_{l})_{k,l\in L}$.
We define an anti-unitary operator $C:h^{2}\rightarrow h^{2}$ by
\begin{equation*}
C(e_{k},e_{l})=(e_{l},e_{k})
\end{equation*}
for all $k,l\in L$. Let $E$ be the the so-called basis projection, \emph{i.e.} the projection of $h^{2}$ onto $h\oplus 0$. Then one has
\begin{equation*}
E+CEC=\idd_{h^{2}}\,.
\end{equation*}
The basic operator in the self-dual approach is
\begin{equation*}
c(z):=a^{\dagger}(ECz)+a(Ez)\in \cb(\ch)
\end{equation*}
for $z\in h^{2}$. Here, an element $(x,0)$ of $h^{2}$ is identified with the
element $x$ of $h$, \emph{i.e.} we set $a(x,0):=a(x)$ for all $x\in h$.
Note further that for all $l\in L$, $c(e_{l},0)=a_{l}$ and $c(0,e_{l})=a_{l}^{\dagger}$.
Finally, $c(z)^{\ast }=c(Cz)$ for any $z\in h^{2}$.

For any closed $C$-invariant subspace $Z$ of $h^{2}$, in \cite{BJL} the authors
consider the von Neumann algebra $(\mathcal{M}(Z), \ch)$, where
\begin{equation*}
\mathcal{M}(Z):=\{c(z) \mid z\in Z\}^{\prime \prime}\,.
\end{equation*}
Their main result is
\begin{equation}
\mathcal{M}(Z)^{\prime }=K\mathcal{M}(Z^{\perp })K^{\ast }\,,  \label{BJL-dua}
\end{equation}
where $K:\ch\rightarrow\ch$ is the unitary operator defined by
\begin{equation}
Kf_{s}=\left\{
\begin{array}{cc}
f_{s}&\text{ \ if }s\text{ has even length}\,,\\
-if_{s}&\text{ \ if }s\text{ has odd length}\,,
\end{array}
\right.  \label{tralK}
\end{equation}
for all $s\in D_{L}$ (see \cite{BJL}, Sections II.B and VII).

We now specialise to
\begin{equation*}
Z:=\overline{\spa(\{(e_{k},0):k\in I\}\cup \{(0,e_{l}):l\in I\})}\,,
\end{equation*}
where one finds
\begin{equation}
\mathcal{M}(Z)=\mathcal{A}(I)\,.  \label{MvsA}
\end{equation}
Furthermore, $\mu _{I}\in \cs(\mathcal{A}(I))$ defined by $\mu _{I}(a):=\left\langle a\zeta ,\zeta \right\rangle$
for all $a\in \mathcal{A}(I)$, is the state determined by the density matrix $\rho _{I}$ in (\ref{roI}).

In the remainder of this section we assume that for any $s\in D_I$
$$
p_s>0\,.
$$

The following proposition is crucial in order to apply our abstract results to the fermionic entangled state. When $I$ is finite, the statement was proved in \cite{Dfer}, Proposition 7.6 (ii).
\begin{proposition}
	\label{FiIsSik}
	The fermionic entangled vector $\zeta$ in (\ref{diag}) is cyclic for $\mathcal{A}(I)$ in $\ch$.
\end{proposition}

\begin{proof}
As noted above, we reduce the matter to infinite countable $I$. Let us first introduce some notation. For a sequence $s=(s_{1},...,s_{n})\in D_{I}$, we write $a_{(s)}:=a_{s_{n}}a_{s_{n-1}}\cdots a_{s_{1}}$.
	If for $s,t\in D_{I}$ all entries of $s$ are also present in $t$ (even in a
	different order), we say that $t$ contains $s$, written as $s\subset t$, and
	in this case we define $t\backslash s\in D_{I}$ as the string consisting of
	those entries of $t$ which are not in $s$ (the order of which is determined
	by our choice of $D_{I}$).

	Now fix $s,t\in D_{I}$, and define
$$
\zeta_{2}:=a_{(t)}\zeta\,.
$$	
If $r\in D_{I}$ does not contain $t$, then by \eqref{c0} the terms $p_{r}^{1/2}f_{r\iota (r)}$ in $\zeta$
are annihilated by $a_{t}$. Therefore, $\zeta_{2}$ consists exactly of all the terms of the form
	\begin{equation*}
	\pm p_{r}^{1/2}f_{(r\backslash t)\iota (r)}\,,
	\end{equation*}
	for $r\in D_{I,t}:=\{q\in D_{I} \mid t\subset q\}$.
	In particular, $p_{t}^{1/2}f_{\iota (t)}$ appears as a term in $\zeta_{2}$. The probabilities not
	appearing in these terms add up to
	\begin{equation*}
	p:=\sum_{r\in D_{I}\backslash D_{I,t}}p_{r}\,.
	\end{equation*}
	So, for any given $\varepsilon >0$, there is a finite subset $F$ of $D_{I,t}$
	such that
	\begin{equation*}
	\sum_{r\in F}p_{r}>1-p-\varepsilon p_{t}\,,
	\end{equation*}
	or equivalently
	\begin{equation*}
	\sum_{r\in D_{I,t}\backslash F}p_{r}<\varepsilon p_{t}\,.
	\end{equation*}
	Let $u\in D_{I}$ be the sequence consisting of all the elements of $I$ which
	appear as an entry in at least one of the sequences in $F$, and define
	\begin{equation*}
	\zeta_3:=a_{(u)}^{\dagger}\zeta_2\,.
	\end{equation*}
	As $\zeta_2$ contains $p_{t}^{1/2}f_{\iota(t)}$, again \eqref{c0} gives that $p_{t}^{1/2}f_{u\iota(t)}$ is a term in $\zeta_3$. Since the terms $\pm p_{r}
	^{1/2}f_{(r\backslash t)\iota(r)}$ in $\zeta_{2}$ such that $r\backslash t$
	contains entries from $u$ are annihilated by $a_{(u)}^{\dagger}$, the
	probabilities $p_{r}$ appearing in the terms of $\zeta_{3}$ are therefore
	exactly those with $r$ containing all entries of $t$, but no further entries
	from any of the sequences in $F$. No such $r$ is in $F$, except possibly for $
	r=t$. As for $v,w\in D_{I}$ the vectors $f_{v\iota(w)}$ are orthonormal, it follows that
	\begin{equation}
\label{xx}
	\left\|\zeta_3-p_t^{1/2}f_{u\iota(t)}\right\|^2 \leq\sum_{r\in
		D_{I,t}\backslash F}p_{r}<\varepsilon p_t\,.
	\end{equation}
	Lastly, we set
	\begin{equation*}
	\zeta_{4}:=a_{(s)}^{\dagger}a_{(u)}\zeta_{3}\,,
	\end{equation*}
	where further terms may be annihilated, but $\zeta_{4}$ contains at least the term $p_{t}^{1/2}f_{s\iota (t)}$. Again the orthonormality gives	\begin{equation*}
	\left\| f_{s\iota (t)}-\frac{1}{p_{t}^{1/2}}\zeta_{4}\right\| ^{2}\leq \frac{1}{p_{t}}\left\| p_{t}^{1/2}f_{u\iota (t)}-\zeta_{3}\right\| ^{2}<\varepsilon\,,
	\end{equation*}
where the second inequality comes from \eqref{xx}. After recalling that vectors of the form $f_{s\iota (t)}$ constitute a basis for $\ch$, and taking into account that $\zeta_4\in \mathcal{A}(I)\zeta$, the thesis is achieved since $\varepsilon$ is arbitrary.
\end{proof}
Let now $\Gamma :\ch\rightarrow\ch$ be the self-adjoint unitary defined by
\begin{equation}
\Gamma f_{s}=\left\{\begin{array}{cc}
f_{s}&\text{ \ if }s\text{ has even length}\,, \\
-f_{s}&\text{ \ if }s\text{ has odd length}\,,
\end{array}
\right.  \label{tralGrad}
\end{equation}
for all $s\in D_{L}$. As usual, it induces a $\mathbb{Z}_{2}$-grading $\g:=\ad_{\G}$ on $\cb(\ch)$, and $\Gamma \mathcal{A}(I)\Gamma =\mathcal{A}(I)$, $\Gamma \zeta =\zeta$. By restriction, $(\mathcal{A}(I),\mu_{I})$ is then equipped with a $\bz_2$-grading leaving $\m_I$ invariant.
Again we take the Klein
transformation w.r.t. $\G$, \emph{i.e.} $\eta :\cb(\ch)\rightarrow \cb(\ch)$, and the corresponding twisted commutant
$\mathcal{A}(I)^{\wr }$. Here, we removed the subscript since no confusion can arise. It turns out that
(\ref{tralK}) is
a specific case of (\ref{Kdef}), and therefore we obtain the following version of (\ref{BJL-dua}).
\begin{proposition}
	\label{verwKomVsKomplement}Given the bijection (\ref{komplBij}) and the $\mathbb{Z}_{2}$-grading of $\mathcal{A}(I)$ obtained from (\ref{tralGrad}),
	we have
	\begin{equation*}
	\mathcal{A}(I)^{\wr }=\mathcal{A}(L\backslash I).
	\end{equation*}
\end{proposition}
\begin{proof}
Indeed, by Proposition \ref{vKomForm}, the thesis follows directly from \eqref{kap&K}, \eqref{BJL-dua} and \eqref{MvsA}. Namely, recalling that $K^2=\G$, one has \begin{equation*}
	\mathcal{A}(I)^{\wr }=\eta (\mathcal{A}(I)^{\prime })=K\mathcal{M}
	(Z)^{\prime }K^{\ast }=\G \mathcal{A}(L\backslash I)\G=\mathcal{A}(L\backslash I)\,.
\end{equation*}
\end{proof}
The proposition above allows us to apply Theorem \ref{alfaGamma+} to the lattice.
\begin{corollary}
	\label{FiIsSkeid}
The fermionic entangled vector $\zeta$ is separating for $\mathcal{A}(I)$ in $\ch$.
\end{corollary}
\begin{proof}
By swapping the roles of $I$ and $L\backslash I$, Proposition \ref{FiIsSik} yields that $\zeta$ is cyclic for $\mathcal{A}(L\backslash I)=\mathcal{A}(I)^{\wr }=\eta (\mathcal{A}(I))^{\prime }$, where the last two equalities follow from Proposition \ref{verwKomVsKomplement} and Proposition \ref{vKomForm}, respectively. Therefore, $\zeta$ is separating for $\eta (\mathcal{A}(I))$. Consider now $a\in\mathcal{A}(I)$ such that $a\zeta=0$. As $K\zeta=\zeta$, from \eqref{kap&K} one has $0=K a\zeta=\eta(a)\zeta$. Then $\eta(a)=0$, and thus $a=0$.
\end{proof}
At last, by
combining the results of this section with those of Section \ref{AfdDualiteit}, we solve the problem described above.
\begin{theorem}
	\label{Oplossing}	
Let $B_{\varphi }(a,b):
	\mathcal{A}(I)\times \mathcal{A}(L\backslash I)\rightarrow \mathbb{C}$
	be the bilinear form given by
	\begin{equation*}
	B_{\varphi }(a,b):=\varphi (ab)=\left\langle ab\zeta ,\zeta \right\rangle\,, \quad a\in \mathcal{A}(I)\,,\,\,\, b\in \mathcal{A}(L\backslash I)\,.
	\end{equation*}
	Then, for any positive and even linear map $\Psi :\mathcal{A}(I)\rightarrow \mathcal{A}(I)$ such
	that $\varphi \circ \Psi (a)=\varphi (a)$ for all $a\in \mathcal{A}(I)$,
	there exists a unique (and necessarily linear) \emph{fermionic dual} map
	\begin{equation*}
	\Psi ^{\varphi }:
	\mathcal{A}(L\backslash I)\rightarrow \mathcal{A}(L\backslash I)
	\end{equation*}
such that
	\begin{equation*}
	B_{\varphi }(\Psi (a),b)=B_{\varphi }(a,\Psi ^{\varphi }(b))\,, \quad \quad a\in \mathcal{A}(I)\,,\,\,\, b\in \mathcal{A}(L\backslash I)\,.
	\end{equation*}
Furthermore, $\Psi ^{\varphi }$ is positive, unital, faithful and even.
In addition:
	
	(i) if $\Psi$ is $n$-positive for some $n\in\bn$ , then $\Psi
	^{\varphi}$ is $n$-positive;
	
	(ii) if $\Psi$ completely positive, then $\Psi^{\varphi}$ is completely positive;
	
	(iii) if $\Psi $ is unital, then $\varphi \circ \Psi ^{\varphi
	}(b)=\varphi (b)$ for all $b\in \mathcal{A}(L\backslash I)$. Moreover, $\Psi ^{\varphi \varphi }:=(\Psi ^{\varphi })^{\varphi }:\mathcal{A}(I)\rightarrow \mathcal{A}(I)$ is well-defined, \emph{i.e.} it exists and is
	uniquely determined by
	\begin{equation*}
	B_{\varphi }(\Psi ^{\varphi \varphi }(a),b)=B_{\varphi }(a,\Psi
	^{\varphi }(b))\,, \quad a\in \mathcal{A}(I)\,,\,\,\, b\in \mathcal{A}(L\backslash I)\,.
	\end{equation*}
	Finally, $\Psi^{\varphi \varphi}=\Psi$.
\end{theorem}
\begin{proof}
	Indeed, applying Proposition \ref{verwrBehoud} and Theorem \ref{alfaGamma+} to $\mathcal{M}=\mathcal{N}=\mathcal{A}(I)$, with $\xi_\m$
	replaced by $\zeta$, one obtains $\Psi ^{\varphi }=\Psi
	^{\wr }$. Thus, exploiting Proposition \ref{FiIsSik}, Proposition \ref{verwKomVsKomplement}, and finally Corollary \ref{FiIsSkeid}, one achieves the thesis.

In particular, as $\Psi
	^{\varphi }$ is positive and even, and since $\varphi \circ \Psi ^{\varphi
	}(b)=\varphi (b)$ for all $b\in \mathcal{A}(L\backslash I)$, swapping the roles of $I$
and $L\backslash I$, $\Psi^{\f\f}$ is the dual map associated to the bilinear form
	\begin{equation*}
	\mathcal{A}(I)\times \mathcal{A}(L\backslash I)\ni (a,b)\mapsto
	\overline{B_{\varphi }(a^{\ast },b^{\ast })}=\left\langle ba\zeta,\zeta\right\rangle\,.
	\end{equation*}
	As a consequence, $\Psi ^{\varphi \varphi }$ is uniquely determined by
$$
\left\langle \Psi ^{\varphi
	}(b)a\zeta,\zeta \right\rangle =\left\langle b\Psi ^{\varphi \varphi
	}(a)\zeta,\zeta \right\rangle
$$
as a positive map, where $a\in \mathcal{A}(I)$ and $b\in \mathcal{A}(L\backslash I)$. Finally,
$$
B_{\varphi
	}(\Psi ^{\varphi \varphi }(a),b)=\overline{\left\langle b^{\ast }\Psi
		^{\varphi \varphi }(a^{\ast })\zeta,\zeta\right\rangle }=
	\overline{\left\langle \Psi ^{\varphi }(b^{\ast })a^{\ast }\zeta,\zeta\right\rangle}
	=B_{\varphi }(a,\Psi ^{\varphi }(b))\,.
$$
\end{proof}
As suggested by (iii) above, one can also look at the ``dual version'' of this theorem by
working in terms of $B_{\varphi }(\Phi ^{\varphi }(a),b)=B_{\varphi
}(a,\Phi (b))$, for $\Phi :\mathcal{A}(L\backslash I)\rightarrow \mathcal{A}
(L\backslash I)$ positive, even and state preserving.

\begin{remark}
	Using the same argument as in \eqref{linkerduaal}, we may as well define the
	bilinear form as $B_{\varphi }(a,b)=\left\langle ba\zeta,\zeta \right\rangle $.
	It leads exactly to the same dual $\Psi ^{\varphi }$ of $\Psi $,
	since $\Psi ^{\varphi }$ is positive in the theorem above.
\end{remark}

\section{Fermionic detailed balance}

\label{AfdFfb}

The theory developed in this paper provides
a natural framework to formulate detailed balance for fermionic systems.

The discussion is spread over two sections, the latter dealing with a more abstract setting. Here, using the Fermi tensor product,
our formulation of fermionic
detailed balance is based on the diagonal state of a compound system.

For the lattice in the finite dimensional case (\emph{i.e.} $I$ finite), fermionic detailed balance has been defined in
\cite{Dfer}. As it appears useful for our definition, in the next lines we briefly recall it,
using the notation of the previous section. Thus, consider a unital positive map
$\Psi :\mathcal{A}(I)\rightarrow \mathcal{A}(I)$, which has to be carried over to $\mathcal{A}(\iota(I))$ in order to have its copy on the latter algebra.
In more detail, if $\varkappa :\mathcal{A}(I)\rightarrow \mathcal{A}(\iota (I))$ is the $\ast$-isomorphism given by
\begin{equation}
\varkappa (a_{l}):=a_{\iota (l)}\,, \quad l\in I\,,   \label{tralKopie}
\end{equation}
one considers
\begin{equation}
\Psi ^{\iota }:=\varkappa \circ \Psi \circ \varkappa ^{-1}:\mathcal{A}(\iota
(I))\rightarrow \mathcal{A}(\iota (I)) \label{dinKop}
\end{equation}
as the copying map of $\Psi$.
Motivated by the seminal papers \cite{DHR1, DHR2}, and standard quantum detailed balance w.r.t. a reversing operation studied
in \cite{FR,DS1}, fermionic detailed balance of $\Psi$ has been defined in
\cite{Dfer} by means of the fermionic entangled pure state $\f$ on the compound system as
\begin{equation}
\label{ffb}
\f(a\Psi^{\iota}(b))=\f(\Psi(a)b)\,, \quad a\in \mathcal{A}(I)\,,\,\,\, b\in \mathcal{A}(\iota (I))\,.
\end{equation}
In the case where $\iota(I):=L\backslash I$, we see by Proposition \ref{verwKomVsKomplement} that \eqref{tralKopie}
and \eqref{dinKop} can respectively be expressed as
\begin{equation}
\varkappa :\mathcal{A}(I)\rightarrow \mathcal{A}(I)^\wr\,, \quad \text{and}\,\,\,\,\, \Psi ^{\iota }:=\varkappa \circ \Psi \circ \varkappa ^{-1}:
\mathcal{A}(I)^\wr\rightarrow \mathcal{A}(I)^\wr\,.
\label{kopIso2}
\end{equation}
Under the same notation of Theorem \ref{Oplossing}, \eqref{ffb} can be therefore formulated as
\begin{equation}
\Psi ^{\varphi }=\Psi ^{\iota }\,,\label{einddinFFB}
\end{equation}
or equivalently
\begin{equation*}
B_{\varphi }(a,\Psi ^{\iota }(b))=B_{\varphi }(\Psi (a),b)
\end{equation*}
for all $a\in \mathcal{A}(I)$ and $b\in \mathcal{A}(\iota (I))$.
We use this as the basis for the following generalization.

For our aim, it is useful to first give a definition of diagonal states directly in the von Neumann algebra setting. The definition in
Section \ref{AfdDiag} for a $W^*$-algebra $\gam$ reduces to the following one by
considering the von Neumann algebra $(\pi_\f(\gam)'',\ch_\f)$, together with the
grading on $\cb(\ch_\f)$ induced by $\G:=V_{\f,\th}$ and $\xi:=\xi_\f$.
\begin{prop}
Let $(\cam,\ch)$ be a von Neumann algebra, together with a selfadjoint unitary $\G$ acting on $\ch$ whose adjoint action leaves $\cam$ globally stable: $\cam=\G\cam\G$. Consider
the corresponding twisted commutant $\cam^{\wr}_\G$ given in Definition \ref{verwKom}.

For each unit vector $\xi\in\ch$,
$$
\mathcal{M}\, \circled{\rm{{\tiny F}}}^{\rm bin}_{\rm max}\mathcal{M}^\wr_\G\ni a\ftp b^{\wr}\mapsto\d_\xi
\big(a\ftp b^{\wr}\big):=\langle ab^{\wr}\xi,\xi\rangle\in\bc
$$
uniquely defines a state. If in addition $\G\xi=\pm\xi$, then $\d_\xi$ is even.
\end{prop}
\begin{proof}
Put $\mathcal{M}^\wr=\mathcal{M}^\wr_\G$. By reasoning as in Proposition 4.1 of \cite{CrF2}, we argue that
$$
\cam\ftp \cam^{\wr}\ni a\ftp b^{\wr}\mapsto ab^{\wr}\in\cb(\ch)
$$
uniquely extends to a nondegenerate representation $\pi$ of $\mathcal{M}\, \circled{\rm{{\tiny F}}}^{\rm bin}_{\rm max}\mathcal{M}^\wr_\G$.
Therefore, $\d_\xi$ uniquely defines a state, by restricting the vector state generated by $\xi$ to the image of $\pi$.

Finally, let $\th$ and $\th^\wr$ be the grading implemented by $\G$ on $\mathcal{M}$ and $\mathcal{M}^\wr$ by restriction, respectively.
Denoting by $\th\, \circled{\rm{{\tiny F}}}\,\th^\wr$ the induced grading on $\mathcal{M}\, \circled{\rm{{\tiny F}}}^{\rm bin}_{\rm max}\mathcal{M}^\wr$ ({\it cf.} \eqref{aupf}),
we easily get for $a\, \circled{\rm{{\tiny F}}}\,b^\wr\in\mathcal{M}\, \circled{\rm{{\tiny F}}}^{\rm bin}_{\rm max}\mathcal{M}^\wr_\G$,
$$
\d_\xi\big((\th\, \circled{\rm{{\tiny F}}}\,\th^\wr)(a\, \circled{\rm{{\tiny F}}}\,b^{\wr})\big):=\langle \G ab^{\wr}\G\xi,\xi\rangle
=\langle ab^{\wr}\G\xi,\G\xi\rangle=\d_{\G\xi}(a\, \circled{\rm{{\tiny F}}}\,b^\wr)\,,
$$
and then $\d_\xi$ coincides with $\d_{\G\xi}$ if $\G\xi=\pm\xi$.
\end{proof}
We can easily deduce that the GNS representation $(\ch_{\d_\xi},\pi_{\d_\xi},\xi_{\d_\xi})$ is $(P\ch, P\pi,\xi)$, where
$P\in\left(\mathcal{M}\bigvee\mathcal{M}^\wr\right)'$ is the cyclic projection onto $\overline{\left(\mathcal{M}\bigvee\mathcal{M}^\wr\right)\xi}\subset\ch$.

Notice that the name ``diagonal state" for $\d_\xi$ is justified from the fact that it can be viewed
as the diagonal state associated to the product state $\langle\,{\bf\cdot}\,\xi,\xi\rangle\lceil_\mathcal{M}\times\langle\,
{\bf\cdot}\,\xi,\xi\rangle\lceil_{\mathcal{M}^\wr}$, provided at least one of the factors is even.

Under the assumptions and notations of the above proposition, we assume further that $\xi\in \ch$ is cyclic and separating for $\mathcal{M}$,
and consider the even state $\m:=\langle\,{\bf\cdot} \xi, \xi\rangle$.
In order to formulate fermionic standard quantum detailed balance in this setting,
we generalise $\varkappa$ in \eqref{tralKopie} by assuming that we have a $\ast $-isomorphism
\begin{equation}
\varkappa :\mathcal{M}\rightarrow \mathcal{M}^{\wr }\,, \label{dinKopie}
\end{equation}
which copies the dynamics. As usual, $\mathcal{M}^{\wr}$ stands for $\mathcal{M}^{\wr }_\G$. For any unital positive map
$\Psi :\mathcal{M}\rightarrow \mathcal{M}$, we obtain
\begin{equation}
\label{alka}
\Psi ^{\varkappa}:=\varkappa\circ \Psi \circ \varkappa^{-1}:
\mathcal{M}^{\wr}\rightarrow \mathcal{M}^{\wr }
\end{equation}
as in (\ref{kopIso2}). Note that if we assume that $\varkappa$ is grading-equivariant, then $\Psi ^{\varkappa}$ is automatically
even when $\Psi$ is, a condition naturally satisfied in the lattice.
Recalling that $\Psi^{\wr}:\mathcal{M}^{\wr } \rightarrow \mathcal{M}^{\wr }$ is the twisted dual map of $\Psi$, we now state the main concept of this section.
\begin{definition}
	\label{algFFB}
Under the above notations, we say that $\Psi$ satisfies \emph{fermionic standard quantum
	detailed balance} (with respect to $\mu$ and $\varkappa$), if
\begin{equation}
\Psi^\wr=\Psi^\varkappa.  \label{affb}
\end{equation}
\end{definition}
\vskip.3cm
Notice that the requirement (\ref{affb}) does not need the assumption that $\Psi$ is even.
The latter is indeed the case we are most interested in, since by Theorem \ref{alfaGamma+} it guarantees that
$\Psi^\wr$ is positive (resp. completely positive) if $\Psi$ is also positive (resp. completely positive). Nevertheless, if condition
(\ref{affb}) is satisfied, then the complete positivity of $\Psi$ by itself entails the same property for $\Psi^\wr$, as a consequence of \eqref{alka}.

The definition above generalises the finite dimensional lattice case (\ref{einddinFFB}), where $\Psi
^{\wr }$ is seen as the abstract version of $\Psi ^{\varphi}$. Furthermore,
after noticing that the unital algebra $\cam\ftp^{\rm bin}_{\rm max} \cam^{\wr}$
replaces the algebra $\ca(L)$, by
viewing $\cam$ and $\cam^{\wr}$ as playing the role of $\ca(I)$  and $\ca(I)^{\wr}=\ca(L\backslash I)$ (see Proposition
\ref{verwKomVsKomplement}) respectively, the diagonal state above defined allows to get a more general version of \eqref{ffb}.
\begin{prop}
\label{ds}
Fermionic standard quantum detailed balance as stated in Definition \ref{algFFB} can equivalently be
formulated as
\begin{equation}
\delta_{\xi}(a\ftp\Psi^\varkappa(b^{\wr}))=\delta_{\xi}(\Psi(a)\ftp b^{\wr})\,, \quad a\in\cam\,,\,\, b^{\wr}\in\cam^\wr\,.  \label{diagffb}
\end{equation}
\end{prop}
\begin{proof}
Indeed, \eqref{affb} implies \eqref{diagffb} by definition of $\d_{\xi}$.
Moreover, Proposition \ref{verwrBehoud} and Theorem \ref{alfaGamma+} immediately give \eqref{affb} if \eqref{diagffb} holds true.
\end{proof}
Arguing as in Theorem \ref{Oplossing} and its
proof, it is natural to introduce the bilinear form
\begin{equation*}
B:\mathcal{M}\times \mathcal{M}^{\wr}\ni(a,b^{\wr})\mapsto B(a,b^{\wr}):=
\left\langle ab^{\wr}\xi ,\xi\right\rangle \in \bc\,.
\end{equation*}
Thus, condition (\ref{affb}) can be expressed as
\begin{equation}
B(a,\Psi ^{\varkappa }(b^{\wr}))=B(\Psi (a),b^{\wr})\,, \quad  a\in \mathcal{M}\,,\,\,\, b^{\wr}\in \mathcal{M}^{\wr }\,. \label{absBilin}
\end{equation}
Also note that the fermionic diagonal state gives exactly the bilinear form $B$ used in
(\ref{absBilin}), \emph{i.e.}
\[
\d_{\xi}(a\ftp b^{\wr})=B(a,b^{\wr})\,, \quad  a\in \mathcal{M}\,,\,\,\, b^{\wr}\in \mathcal{M}^{\wr }\,.
\]

\begin{rem}
Observe that $\delta_{\xi}$ plays the same role as $\varphi$ for the case of the lattice in
Theorem \ref{Oplossing} and identity (\ref{ffb}), and
$\xi$ can be viewed as an abstract version of the entangled vector $\zeta$, in analogy to the case of $\cam\odot \cam'$.
In the latter setting the diagonal state is a generalization of the entangled physical pure
state of a compound system consisting of two copies of
$\cb(\ch)$, with $\ch$ the Hilbert space of the system (see, \emph{e.g.} \cite{DS2}, Section 7).
One could wonder up to what generality the vector $\xi$
is physically an entangled state of the compound system $\mathcal{M}\, \circled{\rm{{\tiny F}}}^{\rm bin}_{\rm max}\mathcal{M}^\wr$. As this is not part of the aims of the present paper, we postpone a more exhaustive treatment to somewhere else.
\end{rem}
When $\theta=\id_{\mathcal{M}}$, \eqref{affb} reduces to the standard quantum detailed balance with respect to a reversing operation $\vartheta$ \cite{FR,DS1}.

In more detail, in this case one takes an involutive $*$
-antiautomorphism $\vartheta :\mathcal{M}\rightarrow \mathcal{M}$ (where ``anti"
simply means $\vartheta (a_{1}a_{2})=\vartheta (a_{2})\vartheta (a_{1})$ for $a_1,a_2\in\mathcal{M}$) such that $\mu \circ \vartheta =\mu$,
and considers $j(a):=J_{\m}a^{\ast }J_{\m}$
for all $a\in \cb(\ch)$, where $J_{\m}$ is the modular conjugation
associated with $\mu$. The standard quantum detailed balance w.r.t. the reversing map $\vartheta$ (also called $\vartheta$-sqdb), is defined by
\begin{equation}
\Psi ^{\vartheta }=\Psi\,,  \label{thetfb}
\end{equation}
where
\begin{equation*}
\Psi ^{\vartheta }:=\vartheta \circ j\circ \Psi ^{\prime }\circ j\circ \vartheta\,.
\end{equation*}
It is easy to see that \eqref{thetfb} is equivalent to
\begin{equation}
\label{tdb}
\Psi ^{\prime }=\Psi ^{\varrho}\,,
\end{equation}
where we have used $\varrho :=j\circ \vartheta :\mathcal{M}\rightarrow \mathcal{M}^{\prime}$ to copy $\Psi$ to $\Psi ^{\varrho}:=\varrho\circ \Psi \circ \varrho^{-1}:
\mathcal{M}^{\prime}\rightarrow \mathcal{M}^{\prime }$. Since now the twisted $\ast$-automorphism is trivial, \eqref{affb} and \eqref{tdb} are analogous.
\begin{example}
In the algebra $\bm_n(\bc)$ one takes $\vartheta $ as the transposition operation
with respect to the basis in which the density matrix is diagonal. Here, $\mathcal{M}:=\bm_n(\bc)\otimes \idd_{\bm_n(\bc)}$, and $\mathcal{M}^{\prime
}=\idd_{\bm_n(\bc)}\otimes \bm_n(\bc)$. Since for any $a\in \bm_n(\bc)$, $j(a\otimes \idd_{\bm_n(\bc)})=\idd_{\bm_n(\bc)}\otimes a^{T}$, with
$a^{T}$ the transpose of $a$ (see for example \cite{DS2}, Section 7), one finds
$$
\varrho(a\otimes \idd_{\bm_n(\bc)})=\idd_{\bm_n(\bc)}\otimes a\,.
$$
After replacing $\idd_{\bm_n(\bc)}\otimes \bm_n(\bc)$ with $\bm_n(\bc)\otimes \idd_{\bm_n(\bc)}$,
one notices $\Psi ^{\varrho}$ is the same physical dynamics as $\Psi $ on $\bm_n(\bc)$.
Therefore, in this instance, copying using $\varrho$ appears physically sensible.
\end{example}
\vskip.3cm

In general we need a copying $\ast $-isomorphism from $\cam$ to the
appropriate ``dual'' von Neumann algebra, which is physically sensible in
concrete examples, and mathematically natural in the abstract framework. In
the case of $\theta $-sqdb, the dual von Neumann algebra is
$\mathcal{M}^{\prime }$, but in the fermionic case it is $\mathcal{M}^{\wr }$.

For a non trivial
$\bz_2$-graded structure of $\mathcal{M}$,
one would attempt to use $\widetilde{\varkappa}:=\eta\circ j\circ \vartheta$ as the copying map.
Then, for even $\Psi$, we could obtain the usual $\vartheta$-sqdb condition \eqref{thetfb} from (\ref{affb}) by Theorem \ref{alfaGamma+}. Namely,
\begin{align*}
\Psi & =\widetilde{\varkappa} ^{-1}\circ \Psi ^{\widetilde{\varkappa}}\circ \widetilde{\varkappa}
=\widetilde{\varkappa} ^{-1}\circ \Psi^{\wr }\circ \widetilde{\varkappa} \\
&=\vartheta \circ j\circ \eta ^{-1}\circ \Psi ^{\wr }\circ
\eta \circ j\circ \vartheta \\
& =\vartheta \circ j\circ \Psi ^{\prime }\circ j\circ \vartheta =\Psi ^{\vartheta}\,.
\end{align*}
Hence, although mathematically natural, this choice of copying the dynamics
is not physically appropriate for a structure equipped with a non trivial
$\bz_2$-grading, as it just reproduces the usual $\vartheta$-sqdb condition.
This is confirmed by the fact that the physically sensible copying $*$-isomorphism
in a lattice given by \eqref{tralKopie}, differs from $\widetilde{\varkappa}$,
as the next example
shows:
\begin{example}
Consider the lattice in the finite dimensional case (\emph{i.e.} $|I|< \infty$).
Let $\vartheta:\mathcal{A}(I)\rightarrow \mathcal{A}(I)$ be the transposition as mentioned above,
that is $\vartheta$ is the $*$-antiautomorphism such that for any $s,t\in D_I$
$$
\langle \vartheta(a)f_s,f_t\rangle:=\langle af_t,f_s\rangle\,, \quad a\in \mathcal{A}(I)\,.
$$
This means that $\vartheta(a_l)=a^\dag_l$, for $l\in I$. Simply taking
$I:=\{1,2\}$, $\iota (1)=3$ and $\iota (2)=4$, when $\mu $ is the normalised
trace, one has
$$
\zeta=\frac{1}{2}(f_{\emptyset}+f_{(1,3)}+f_{(2,4)}+f_{(1,2,3,4)})\,.
$$
Moreover, $j(a_1)\in \pi_{\m}(\mathcal{A}(I))^{\prime}$ and
$$
j(a_1)\zeta=Ja_1^\dag J\zeta= a_1\zeta\,.
$$
Notice that
\begin{align*}
a_1\zeta&= \frac{1}{2}(f_{(3)}+f_{(2,3,4)})\\
&=\frac{1}{2}a^\dag_3(f_{\emptyset}+f_{(1,3)}-f_{(2,4)}-f_{(1,2,3,4)}) \\
&=a^\dag_3(a_4a^\dag_4-a^\dag_4a_4)\Gamma\zeta\,.
\end{align*}
Now $a^\dag_3(a_4a^\dag_4-a^\dag_4a_4)\Gamma\in \pi_\mu(\mathcal{A}(I))^{\prime}$,
since $\Gamma^2=\idd_{\ch_{I\cup \iota(I)}}$, and $\pi_\mu(\mathcal{A}(I))^{\prime}$ is
generated by $\{a_3\Gamma, a_4\Gamma\}$. As $\zeta$ is separating, one finds
$$
j(a_1)=a^\dag_3(\idd_{\ch_{I\cup \iota(I)}}-2a^\dag_4a_4)\Gamma\,.
$$
As a consequence,
$$
j\circ \vartheta(a_1)=j(a_1^\dag)=\Gamma(\idd_{\ch_{I\cup \iota(I)}}-2a^\dag_4a_4)a_3
$$
and thus
$$
\varkappa (a_{1})=a_{3}\neq \imath (\idd_{\ch_{I\cup \iota(I)}}-2a_{4}^{\dag}a_{4})a_{3}
=\widetilde{\varkappa} (a_{1})\,.
$$
Here, \eqref{tralKopie} is the physically sensible way of copying dynamics, as it comes directly from the map $\iota$ copying the lattice sites.
\end{example}
\vskip.3cm

Therefore, when
dealing with fermionic detailed balance, we assume
the presence of an abstract copying $\ast$-isomorphism $\varkappa: \mathcal{M}\rightarrow \mathcal{M}^{\wr}$.

\section{Fermionic detailed balance for abstract $C^*$-systems}
\label{AfdCffb}

Our goal here is to give a definition of fermionic detailed
balance for abstract $C^*$-algebraic systems. It is reached by again using the notion of diagonal state studied in Section \ref{AfdDiag}.

Let $(\ga,\theta)$ be a unital $\bz_2$-graded $C^*$-algebra, and consider an even state $\varphi\in \cs(\ga)$ such that
$s_\varphi\in Z(\mathfrak{A}^{\ast\ast})$. If $\big(\ch_{\f},\pi_{\varphi},V_{\f,\th},\xi_{\f}\big)$
is the GNS covariant representation associated to $\f$, consider the von Neumann algebra $\cam:=\pi_{\varphi}(\mathfrak{A})^{\prime\prime}$
with normal faithful state $\mu$ given by $\mu(a):=\left\langle a\xi_{\varphi},\xi_{\varphi}\right\rangle$,
$a\in \cam$. Let $\gamma:=\ad_{V_{\f,\th}}\lceil_\cam$ be the $\mathbb{Z}_{2}$-grading of $\cam$ endowed with $\theta$, and take the corresponding Klein transformation
\[
\eta:\cb(\ch_{\varphi})\rightarrow \cb(\ch_{\varphi})
\]
as defined in (\ref{eta}). Recall that $\mu\circ\gamma=\mu$.

As in the previous section,
we consider a unital positive map $\Psi:\cam\rightarrow \cam$
as a (dissipative) dynamics on $\cam$, and assume that we have a
copying $*$-isomorphism $\varkappa:\cam\rightarrow \cam^{\wr}$
as in (\ref{dinKopie}) such that fermionic standard quantum detailed balance
(\ref{diagffb}) is satisfied, \emph{i.e.}
\[
\left\langle a\alpha^{\varkappa}(b^{\wr})\xi_{\varphi},\xi_{\varphi}\right\rangle
=\left\langle \alpha(a)b^{\wr}\xi_{\varphi},\xi_{\varphi}\right\rangle\,, \quad a\in \cam\,,\,\, b^{\wr}\in \cam^{\wr}\,.
\]
If $\cam^{\circ}$ is the opposite
algebra of $\cam$, then as in the previous section we take $j(a):=J_{\varphi}a^*J_{\varphi}$ for any $a\in \cb(\ch_{\varphi})$, where $J_\varphi$ is as
usual the modular conjugation associated to $\varphi$. Notice that $j$ realises a $*$-isomorphism between the $W^*$-algebras $\cam^{\circ}$ and $\cam'$.

Since from Proposition \ref{vKomForm}
it follows that $\eta^{-1}\circ\varkappa(\cam)=\cam'$, we consider the $\ast$-isomorphism
\begin{equation*}
\varsigma:=j\circ\eta^{-1}\circ\varkappa:\cam\rightarrow \cam^{\circ}\,.
\end{equation*}
Letting $\Psi^{\varsigma}:=\cam^{\circ}\rightarrow \cam^{\circ}$ be given by
\[
\Psi^{\varsigma}:=\varsigma\circ\Psi\circ\varsigma^{-1}\,,
\]
the detailed balance above can be rewritten as
\begin{equation}
\left\langle a(\eta\circ j\circ\Psi^{\varsigma})(b^{\circ})\xi_{\varphi},
\xi_{\varphi}\right\rangle
=\left\langle \Psi(a)(\eta\circ j)(b^{\circ})\xi_{\varphi},
\xi_{\varphi}\right\rangle    \label{ffbTeenAlg}
\end{equation}
for all $a\in \cam$ and $b^{\circ}\in \cam^{\circ}$.

Thus, we model the $C^*$-algebraic formulation of fermionic detailed balance on
(\ref{ffbTeenAlg}). Indeed, we take a unital positive map
\[
\F:\mathfrak{A}\rightarrow\mathfrak{A}
\]
as the dynamics, and assume we are given the $*$-isomorphism
$$
\r:\mathfrak{A}\rightarrow\mathfrak{A}^{\circ}\,,
$$
which serves as a necessarily abstract copying map.
Then, we present the following
\begin{definition}
\label{dbfz}
We say that $\F$ satisfies \emph{fermionic standard quantum detailed balance} (with respect to $\varphi$ and $\r$) if
\begin{equation}
\delta_{\varphi}(a\ftp \F^{\r}(b^{\circ}))=
\delta_{\varphi}(\F(a)\ftp b^{\circ}) \label{ffbC*-alg}
\end{equation}
for all $a\in\mathfrak{A}$ and $b\in\mathfrak{A}^{\circ}$, where
$\delta_{\varphi}$ is the diagonal state on $\mathfrak{A\ftp_{\max} A}^{\circ}$ given in
Theorem \ref{gprdig},
and
\[
\F^{\r}:=\r\circ\F\circ\r^{-1}:
\mathfrak{A}^{\circ}\rightarrow\mathfrak{A}^{\circ}\,.
\]
\end{definition}
\vskip.3cm
Notice that \eqref{ffbC*-alg} is a version of (\ref{diagffb}) expressed in terms of the diagonal state \eqref{fidiag}.
The main structure involved here is the bilinear form
\[
B_{\varphi}:\mathfrak{A}\times\mathfrak{A}^{\circ}\rightarrow\mathbb{C}
\]
given by
\[
B_{\varphi}(a,b^{\circ}):=\delta_{\varphi}(a\circled{\rm{{\tiny F}}} \,b^{\circ})\,, \quad a\in\ga\,,\,\, b^{\circ}\in \ga^{\circ}\,,
\]
and (\ref{ffbC*-alg}) can be given in terms of $B_{\varphi}$, in analogy to (\ref{absBilin}) in a natural way.\footnote{Even if Definition \ref{dbfz} is meaningful for general positive maps $\F$ and the $*$-isomorphisms $\r$, for physical applications these will be automatically even ({\it }grading-equivariant).}

We end by noticing that the above definition of detailed balance on one hand provides a unifying definition of {\it quantum detailed balance} in a very general situation. For example, it covers the usual tensor product case for which the grading is trivial.
On the other hand, Definition \ref{dbfz} clarifies the natural appearance of the diagonal state under the additional condition of centrality of the support of the involved state.\footnote{Concerning the classical ({\it i.e.} commutative) case, the definition of the ``diagonal measure" associated to the product measure is always meaningful because of the simple fact that the support of every state in the bidual is automatically central. This is no longer true in the quantum situation.}

For the previous work on the detailed balance, the reader is referred to \cite{Ag, Al, Dfer, DS1, DS2, FR, M} and the references cited therein.
The reader is also referred to \cite{BCM, BCM2, D, D2, D3, Fid, F27a} for the role of such a ``diagonal quantum measure" in quantum ergodic theory and the theory of joinings.

Finally, we would like to mention the natural connections between detailed balance and the notion of KMS-symmetric semigroups, see {\it e.g.} \cite{AC, BQ2, GL}.

\section*{Acknowledgments}
The authors acknowledge the following institutions:
\begin{itemize}
\item (V.C.) the grant ``Probabilit\`{a} Quantistica e Applicazioni", University of Bari;
\item (V.C. \& F. F.) Italian INDAM-GNAMPA;
\item (R.D.) the National Research Foundation of South Africa,
and the DST-NRF Centre of Excellence in Mathematical and Statistical Sciences;
\item (F.F) project ``Sustainability-OAAMP", University of Rome Tor Vergata, CUPE81I18000070005;
\item (F.F) Italian MIUR Excellence Department Project awarded to the Department of Mathematics, University of Rome Tor Vergata, CUP E83C18000100006.
\end{itemize}

\end{document}